\documentclass[reqno,11pt]{amsart}
\numberwithin{equation}{section}

\usepackage[utf8]{inputenc}
\usepackage[english]{babel}
\usepackage{amsmath}
\usepackage{amsfonts}
\usepackage{amssymb}
\usepackage{amsthm}
\usepackage{accents}
\usepackage{cite}
\usepackage{graphicx}
\usepackage{color}
\usepackage{enumitem}
\usepackage{mathrsfs}
\usepackage{mathtools}
\usepackage{geometry}
\usepackage[normalem]{ulem}

\theoremstyle{plain}
\newtheorem{theorem}{Theorem}[section]

\newtheorem{prop}[theorem]{Proposition}
\newtheorem{lemma}[theorem]{Lemma}

\newtheorem{example}[theorem]{Example}

\newtheorem{corollary}[theorem]{Corollary}
\theoremstyle{remark}

\title[Eventual Monotonicity of the Willmore Energy]
{Eventual Monotonicity of the Willmore Energy under Mean Curvature Flow}

\author[Y.~Chen]{Yuyue Chen}

\address{Yuyue Chen, School of Mathematical Sciences, Universiti Sains Malaysia,
11800 USM Penang, Malaysia}
\email{1634062227@qq.com}

\author[S.~Wu]{Shiying Wu}

\address{Shiying Wu, School of Mathematics, Guangxi University, Nanning, Guangxi 530004, PR China}
\email{sywu@st.gxu.edu.cn}

\begin{document}

\begin{abstract}
We establish the eventual monotonicity of the Willmore deficit \( \int_{\Sigma}H^2\,d\mu-16\pi \) along the mean curvature flow of smooth closed embedded strictly convex surfaces in \( \mathbb R^3 \), thereby complementing Huisken's exponential convergence in \cite{huisken1984} with an eventual monotonicity property. As a consequence, the Hawking mass is eventually nondecreasing along the contracting mean curvature flow, in contrast with the usual inclusion monotonicity expected of the Hawking mass. We also discuss totally umbilical solutions in three-dimensional space forms, where the monotonicity directions of the Willmore energy and the Hawking mass depend on the sign of the ambient sectional curvature.
\end{abstract}
\maketitle

\section{Introduction and Main Results}
Huisken's classical convergence theorem \cite{huisken1984} shows that a smooth closed embedded strictly convex hypersurface evolving by mean curvature flow becomes asymptotically spherical after the standard normalization.
We first formulate the mean curvature flow and its normalization. Let \( \Sigma_0\subset\mathbb R^3 \) be a smooth closed embedded strictly convex surface, and let \( \mathbf F:\Sigma_0\times[0,T)\longrightarrow\mathbb R^3 \) satisfy
\begin{equation}
\label{eq:intro-MCF}
\partial_t\mathbf F=-H\nu,
\qquad
\mathbf F(\cdot,0)=\mathbf F_0.
\end{equation}
Set \( \Sigma_t:=\mathbf F(\Sigma_0,t) \) for \( t\in[0,T) \).
Huisken's theorem shows that the flow remains strictly convex and shrinks to a single point \( p\in\mathbb R^3 \) as \( t \to T \). After translating \( p \) to the origin and applying Huisken's area-preserving normalization, we write
\[
\tilde\Sigma_{\tilde t}
=
\tilde{\mathbf F}(\Sigma_0,\tilde t),
\qquad
\tilde{\mathbf F}(\cdot,\tilde t)
=
\psi(t)\bigl(\mathbf F(\cdot,t)-p\bigr).
\]
Here \( \psi(t)>0 \) is chosen so that \( |\tilde\Sigma_{\tilde t}|=|\Sigma_0| \), and the normalized time is defined by \( \tilde t=\int_0^t\psi^2(\tau)\,d\tau \).
As \( t\to T \), we have \( \tilde t\to\infty \). The normalized immersion satisfies
\begin{equation}
\label{eq:intro-normalized-MCF}
\partial_{\tilde t}\tilde{\mathbf F}
=
-\tilde H\tilde\nu
+
\left(
\frac{1}{2|\tilde\Sigma_{\tilde t}|}
\int_{\tilde\Sigma_{\tilde t}}
\tilde H^2\,d\tilde\mu
\right)
\tilde{\mathbf F}.
\end{equation}
Under this normalization, the normalized surfaces converge smoothly to a round sphere as \( \tilde t \to \infty \) \cite{huisken1984}.

Related normalization mechanisms also arise in volume-preserving and constrained curvature flows, where suitable global terms are introduced to control the scale of the evolving hypersurfaces and to study their asymptotic convergence; see, for example, \cite{Huisken1987VolumePreserving,AndrewsWei2021,AndrewsChenWei2021}.

A recurring feature of Huisken's convergence argument is the exponential decay of geometric quantities measuring the deviation from roundness. In particular, for the normalized flow \( \tilde\Sigma_{\tilde t} \), his estimates give
\begin{equation}
\label{eq:intro-huisken-defect}
\int_{\tilde\Sigma_{\tilde t}}
\left(
|\tilde A|_{\tilde g}^2
-
\frac12\tilde H^2
\right)
\,d\tilde\mu
\le
Ce^{-\delta\tilde t},
\end{equation}
for some constants \( C<\infty \) and \( \delta>0 \). Such exponential decay of geometric defects is not specific to mean curvature flow; analogous exponential decay estimates for curvature defects also arise in the normalized Ricci flow near constant-curvature metrics. This suggests the broader question of whether exponentially decaying geometric defects may acquire an eventual monotonicity as the flow approaches its limiting geometry. In the present paper, we address this question for the Willmore deficit under mean curvature flow and prove that the answer is affirmative.

For surfaces in \( \mathbb R^3 \), the quantity in \eqref{eq:intro-huisken-defect} is directly related to the Willmore energy. Let \( \Sigma\subset\mathbb R^3 \) be a smooth closed embedded strictly convex surface. We define its Willmore deficit by
\begin{equation}
\label{eq:intro-W-definition}
W(\Sigma) -16\pi
:=
\int_\Sigma H^2\,d\mu -16\pi,
\end{equation}
where \( H \) is the mean curvature and \( d\mu \) is the induced area element. The classical Willmore inequality states that
\begin{equation}
\label{eq:intro-willmore-inequality}
W(\Sigma) -16\pi\ge0,
\end{equation}
with equality if and only if \( \Sigma \) is a round sphere. For the evolving surfaces, we write \( W(t):=W(\Sigma_t), \tilde W(\tilde t) := W(\tilde\Sigma_{\tilde t}). \) Throughout the paper, we use the convention that the unit sphere \( S^2\subset\mathbb R^3 \) has mean curvature \( H=2 \). If \( \tilde k_{1,\tilde t} \) and \( \tilde k_{2,\tilde t} \) are the principal curvatures of \( \tilde\Sigma_{\tilde t} \), then \( |\tilde A|^2-\frac12\tilde H^2 = \frac12(\tilde k_{1,\tilde t}-\tilde k_{2,\tilde t})^2. \) By the Gauss equation and the Gauss--Bonnet theorem,
\begin{equation}
\label{eq:intro-W-defect-identity}
\begin{aligned}
\tilde W(\tilde t)-16\pi
&=
\int_{\tilde\Sigma_{\tilde t}} (\tilde k_{1,\tilde t}-\tilde k_{2,\tilde t})^2\,d\mu_{\tilde t}\\
&=
2\int_{\tilde\Sigma_{\tilde t}}
\left(
|\tilde A|^2-\frac12\tilde H^2
\right)
\,d\mu_{\tilde t}.
\end{aligned}
\end{equation}
Consequently, in dimension two, Huisken's estimate \eqref{eq:intro-huisken-defect} implies
\begin{equation}
\label{eq:intro-huisken-W-decay}
\tilde W(\tilde t)-16\pi
\le
Ce^{-\delta\tilde t}.
\end{equation}
In particular, \( \tilde W(\tilde t)-16\pi \to 0 \) as \( \tilde t \to \infty \).  We refer to \cite{willmore1965,Simon1993,KuwertSchatzle2001,KuwertSchatzle2002} for further background on the Willmore functional and related variational and flow problems.

The issue is whether the Willmore deficit eventually approaches zero monotonically. Despite the exponential decay in \eqref{eq:intro-huisken-W-decay}, monotonicity is not apparent from the evolution equation itself. Indeed, the standard evolution equations for the induced geometry give the following identity; see, for example, \cite{huisken1984}.
\begin{equation}
\label{eq:intro-W-evolution}
\frac{d}{dt}W(t)
=
-2\int_{\Sigma_t}|\nabla H|^2\,d\mu_t
+
2\int_{\Sigma_t}
H^2
\left(
|A|^2-\frac12H^2
\right)d\mu_t.
\end{equation}
If \( k_{1,t} \) and \( k_{2,t} \) are the principal curvatures, then \( |A|^2-\frac12H^2 = \frac12(k_{1,t}-k_{2,t})^2 \), and hence
\begin{equation}
\label{eq:intro-W-evolution-principal}
\frac{d}{dt}W(t)
=
-2\int_{\Sigma_t}|\nabla H|^2\,d\mu_t
+
\int_{\Sigma_t}
H^2(k_{1,t}-k_{2,t})^2\,d\mu_t.
\end{equation}
Since the Willmore energy is invariant under translations and constant dilations of surfaces in \( \mathbb R^3 \), the normalization gives \( W(t) = \tilde W(\tilde t(t)). \) Therefore
\[
\frac{d}{dt}W(t)
=
\frac{d\tilde t}{dt}
\frac{d}{d\tilde t}
\tilde W(\tilde t(t))
=
\psi^2(t)
\frac{d}{d\tilde t}
\tilde W(\tilde t(t)).
\]
In particular, the derivatives of the normalized and unnormalized Willmore energies have the same sign. Thus the evolution contains two terms of opposite sign: the first measures spatial oscillation of the mean curvature, whereas the second measures the failure of the surface to be umbilical.

Our main result shows that the negative gradient term dominates once the flow enters its asymptotically spherical regime.

\begin{theorem}
\label{main.thm}
The following statements hold.

\smallskip
\noindent\textbf{(i) Unnormalized mean curvature flow.}
Let \( \{\Sigma_t\}_{0\le t<T} \) be the smooth mean curvature flow starting from a smooth closed embedded strictly convex surface \( \Sigma_0\subset\mathbb R^3 \). Then there exists \( t_0\in(0,T) \), depending on \( \Sigma_0 \), such that
\begin{equation}
\label{eq:intro-eventual-W-monotonicity}
\frac{d}{dt}W(t)\le0
\qquad
\text{for all }t\in[t_0,T).
\end{equation}

\smallskip
\noindent\textbf{(ii) Huisken's area-preserving normalized mean curvature flow.}
Let \( \{\tilde\Sigma_{\tilde t}\}_{\tilde t\ge0} \) be a smooth strictly convex solution of Huisken's area-preserving normalized mean curvature flow in \( \mathbb R^3 \), after translating the extinction point to the origin. Let \( R>0 \) be the radius of the limiting round sphere. Then there exists \( \tilde T_0<\infty \) such that
\[
\frac{d}{d\tilde t}\tilde W(\tilde t)\le0
\qquad
\text{for all }\tilde t\ge\tilde T_0.
\]
\end{theorem}

The key ingredient is a quantitative coercivity estimate near the round sphere, modulo translation and dilation. More precisely, after translating and scaling a sufficiently small radial graph so that its spherical harmonic components of degrees \( l=0 \) and \( l=1 \) vanish, we prove that, for every \( \delta>0 \),
\begin{equation}
\label{eq:intro-geometric-coercivity}
\int_\Gamma|\nabla_\Gamma H_\Gamma|_{g_\Gamma}^2\,d\mu_\Gamma
\ge
(4-\delta)
\int_\Gamma(k_{1,\Gamma}-k_{2,\Gamma})^2\,d\mu_\Gamma,
\end{equation}
provided that \( \Gamma \) is sufficiently close to the unit sphere. Combining this estimate with Huisken's convergence theorem and \eqref{eq:intro-W-evolution-principal} gives the eventual monotonicity in Theorem~\ref{main.thm}.

The main theorem also has a direct interpretation from the gradient-flow viewpoint. Mean curvature flow is the negative \( L^2 \)-gradient flow of the area functional. If \( \mathcal A(t):=|\Sigma_t| \), then \( \mathcal A'(t) = -\int_{\Sigma_t}H^2\,d\mu_t = -W(t). \) Thus the Willmore energy is precisely the instantaneous area dissipation rate. Theorem~\ref{main.thm} therefore shows that this dissipation rate itself becomes monotone in the asymptotically spherical regime. Equivalently, the area becomes eventually convex as a function of time.

\begin{corollary}[Eventual convexity of the area]
\label{cor:eventual-area-convexity}
Let \( \{\Sigma_t\}_{0\le t<T} \) be as in Theorem~\ref{main.thm}, and let \( \mathcal A(t):=|\Sigma_t|. \) Then there exists \( t_0\in(0,T) \) such that
\[
\mathcal A''(t)\ge0
\qquad
\text{for all }t\in[t_0,T).
\]
\end{corollary}

The Willmore energy is also closely connected with the Hawking mass.
For a smooth closed surface \( \Sigma \) in a three-dimensional Riemannian manifold,
\begin{equation}
\label{eq:intro-Hawking-mass}
m_H(\Sigma)
:=
\sqrt{\frac{|\Sigma|}{16\pi}}
\left(
1-\frac{1}{16\pi}\int_\Sigma H^2\,d\mu
\right),
\qquad
|\Sigma|:=\int_\Sigma d\mu.
\end{equation}

\begin{corollary}
\label{cor:hawking-mass-euclidean-late-time-intro}
Let \( \{\Sigma_t\}_{0\le t<T} \) be as in Theorem~\ref{main.thm}. Then there exists \( t_0\in(0,T) \) such that
\begin{equation}
\label{eq:intro-Hawking-monotonicity}
\frac{d}{dt}m_H(\Sigma_t)\ge0
\qquad
\text{for all }t\in[t_0,T).
\end{equation}
\end{corollary}

The Hawking mass was introduced by Hawking \cite{Hawking1968} and, in the time-symmetric Riemannian setting, is closely related to the Geroch mass \cite{Geroch1973,HuiskenIlmanen2001}.
Relations among the Hawking mass, Willmore-type curvature integrals, and global geometric inequalities also appear in the work of Bray and Miao \cite{BrayMiao2008}, where the capacity of a surface is estimated in terms of its area and Willmore functional, with applications to the Hawking mass.
The search for a satisfactory notion of quasi-local mass is guided by a number of basic geometric and physical desiderata, already emphasized by Christodoulou and Yau \cite[p.~9]{ChristodoulouYau1988}.

Among these, one particularly desirable property is monotonicity under inclusion, meaning that a larger region should have a greater mass \cite[p.~25]{BraySchoen1999}.
However, the standard Hawking mass has no clear monotone property under inclusion \cite[p.~1885]{MondinoTempletonBrowne2022}.
Under suitable geometric assumptions, the Hawking mass is nondecreasing along inverse mean curvature flow, a property known as Geroch monotonicity and used by Huisken and Ilmanen in their proof of the Riemannian Penrose inequality \cite{Geroch1973,HuiskenIlmanen2001}.
In contrast, Corollary~\ref{cor:hawking-mass-euclidean-late-time-intro}, which follows from Theorem~\ref{main.thm}, shows that the Hawking mass is eventually nondecreasing along the contracting Euclidean mean curvature flow.
Since the enclosed regions shrink as the flow evolves, this monotonicity is opposite to the expected monotonicity under inclusion.
In particular, whenever the Hawking mass increases strictly along the flow, the corresponding nested regions provide a violation of inclusion monotonicity.

We also examine totally umbilical mean curvature flow in three-dimensional space forms to illustrate the role of the ambient curvature. In this setting, the Willmore energy satisfies \( W'(t)=4cW(t) \), so its monotonicity is determined by the sign of the sectional curvature \( c \).

\noindent{\bf Structure of the paper.} The rest of the paper is organized as follows. Section~2 introduces the notation and geometric preliminaries. Section~3 develops the spectral and geometric estimates for small radial graphs, and Section~4 establishes a gauge-fixing procedure that eliminates the low spherical harmonic modes by translation and scaling. Section~5 proves the main coercive estimate and records the relevant transformation laws. The main theorem is proved in Section~6. Finally, Section~7 discusses totally umbilical examples in space forms and the Hawking-mass consequences.

\noindent{\bf Future Directions.} The normalized surfaces converge smoothly, in fact exponentially, to a round sphere \cite{huisken1984}. Huisken's convergence argument exhibits a hierarchy of geometric quantities measuring deviation from sphericity that decay exponentially along the normalized flow. For example, he proves \eqref{eq:intro-huisken-defect} together with exponential decay of the mean-curvature variance; see \cite[Lemmas~10.1--10.2, pp.~262--263]{huisken1984}. He further establishes exponential decay of the trace-free curvature defect and of higher covariant derivatives of the second fundamental form; see \cite[Theorem~10.5 and Lemma~10.6, p.~265]{huisken1984}. The eventual monotonicity established here raises the broader question of whether the asymptotically spherical regime of the classical mean curvature flow carries a more general monotonicity structure: namely, whether natural geometric defect functionals that decay to zero as the normalized flow approaches the sphere tend, under suitable conditions, to become monotone at sufficiently late times. Huisken's exponentially decaying quantities provide a natural family on which to investigate this possibility.

A suggestive analogue occurs in the normalized Ricci flow on closed surfaces. Struwe considers the Calabi energy, namely the \(L^2\)-deviation of the Gaussian curvature from the limiting constant curvature; see \cite[equation~(55), p.~270]{Struwe2002}. Once the flow enters the asymptotic constant-curvature regime, he obtains a differential inequality that forces the Calabi energy to decrease exponentially; see \cite[p.~271]{Struwe2002}. This example suggests that the relation between exponential convergence and eventual monotonicity may not be peculiar to mean curvature flow. More generally, both phenomena may reflect the same asymptotic stability mechanism near a canonical limiting geometry, governed by the decay of perturbative modes and the spectral gap of the linearized flow.

Another direction is to study the Hawking mass in non-Euclidean ambient manifolds. Such manifolds are more relevant to the Riemannian initial data arising from physical spacetimes, and additional ambient-curvature terms enter the evolution. It would be interesting to determine whether suitable curvature assumptions lead to a corresponding late-time monotonicity of the Hawking mass, possibly with a direction more naturally adapted to the ambient geometry and its physical interpretation.

\section{Preliminaries}
We collect here the notation and geometric conventions used throughout the paper. These conventions are chosen so that the later spherical expansions, curvature commutators, and scaling arguments can be compared without further changes of sign or normalization.

Throughout the paper, all surfaces are assumed to be smooth, closed, and oriented unless otherwise stated. We omit the subscripts \( t \) and \( \tilde t \) whenever no confusion can arise. \( C \) denotes a positive constant depending only on the standard round metric \( \bar g \) on \( S^2 \), whose value may change from line to line.

Let \( \Sigma\subset\mathbb R^3 \) be an embedded surface with outer unit normal \( \nu \). We write
\[
g=\{g_{ij}\},
\qquad
A=\{h_{ij}\},
\qquad
H=g^{ij}h_{ij},
\qquad
|A|_g^2=g^{ij}g^{kl}h_{ik}h_{jl},
\]
and denote by \( \nabla,\, \Delta,\, d\mu \) the Levi--Civita connection, Laplace--Beltrami operator, and area element of \( \Sigma \), respectively.

For the normalized surface \( \tilde\Sigma \), we use the corresponding notation
\[
\tilde g=\{\tilde g_{ij}\},
\qquad
\tilde A=\{\tilde h_{ij}\},
\qquad
\tilde H=\tilde g^{ij}\tilde h_{ij},
\qquad
|\tilde A|_{\tilde g}^2
=
\tilde g^{ij}\tilde g^{kl}\tilde h_{ik}\tilde h_{jl},
\]
together with \( \tilde\nabla,\, \tilde\Delta,\, \tilde\nu,\, d\tilde\mu. \)

By \( (\cdot,\cdot) \), we denote the standard Euclidean inner product on \( \mathbb R^3 \). Let \( x=(x^1,x^2) \) be local coordinates on \( \Sigma_0 \). For \( \mathbf F:\Sigma_0\times[0,T)\longrightarrow\mathbb R^3 \), the induced metric and the second fundamental form on \( \Sigma_t \) are
\[
g_{ij}(x,t)
=
\left(
\frac{\partial \mathbf F(x,t)}{\partial x^i},
\frac{\partial \mathbf F(x,t)}{\partial x^j}
\right),
\qquad
h_{ij}(x,t)
=
-\left(
\nu(x,t),
\frac{\partial^2 \mathbf F(x,t)}
{\partial x^i\partial x^j}
\right).
\]

Set
\[
\partial_i:=\frac{\partial}{\partial x^i},
\qquad
\mathbf F_i:=\partial_i\mathbf F,
\qquad
\mathbf F_{ij}:=\partial_i\partial_j\mathbf F.
\]

For a smooth function \( f \) on \( \Sigma \),
\[
\nabla f
=
g^{ij}\frac{\partial f}{\partial x^j}\partial_i,
\qquad
(\nabla f)^i
=
g^{ij}\frac{\partial f}{\partial x^j}.
\]
For a vector field \( X=X^i\partial_i \), its divergence is
\[
\operatorname{div}_\Sigma X
=
\nabla_iX^i
=
\frac{\partial X^i}{\partial x^i}
+
\Gamma^i_{ik}X^k
=
\frac{1}{\sqrt{\det g}}
\frac{\partial}{\partial x^i}
\left(
\sqrt{\det g}\,X^i
\right).
\]
Accordingly,
\[
\Delta f
:=
\operatorname{div}_\Sigma(\nabla f)
=
g^{ij}
\left(
\frac{\partial^2f}{\partial x^i\partial x^j}
-
\Gamma^k_{ij}
\frac{\partial f}{\partial x^k}
\right),
\]
or equivalently,
\[
\Delta f
=
\frac{1}{\sqrt{\det g}}
\frac{\partial}{\partial x^i}
\left(
\sqrt{\det g}\,
g^{ij}
\frac{\partial f}{\partial x^j}
\right).
\]

We denote by \( A^\circ := A-\frac{H}{2}g \) the trace-free second fundamental form. The mean curvature is \( H=g^{ij}h_{ij}=k_1+k_2 \), where \( k_1,k_2 \) are the principal curvatures. With the above sign convention, \( \Delta\mathbf F=-H\nu \), and the mean curvature flow is \( \partial_t\mathbf F = \Delta\mathbf F. \)

Throughout the paper, \( \bar g \) denotes the standard round metric on \( S^2 \), and \( d\omega \) denotes its area element. We write \( \bar\nabla,\, \bar\Delta,\, \bar\nabla^2 \) for its Levi--Civita connection, Laplace--Beltrami operator, and Hessian, respectively. For a symmetric \( 2 \)-tensor \( T \) on \( S^2 \), we write \( T^{\circ_{\bar g}} := T-\frac12(\operatorname{tr}_{\bar g}T)\bar g \) for its trace--free part with respect to \( \bar g \). For a radial graph \( \Gamma = \{(1+v(\omega))\omega:\omega\in S^2\} \), we denote the corresponding geometric quantities by
\[
g_\Gamma,\qquad
\nabla_\Gamma,\qquad
H_\Gamma,\qquad
A^\circ_\Gamma,\qquad
k_{1,\Gamma},\qquad
k_{2,\Gamma},\qquad
d\mu_\Gamma.
\]
Whenever such quantities are compared with tensors on \( (S^2,\bar g) \), they are understood to be pulled back by the radial parameterization \( F:S^2\to\Gamma, \, F(\omega)=(1+v(\omega))\omega. \)

We use \( \langle\cdot,\cdot\rangle \) to denote the relevant inner product when the underlying space is clear from the context. For the standard geometric conventions used below, see, for example, \cite{Lee2018Riemannian}. For vector fields \( X,Y,Z \) on \( \Sigma \), our curvature convention is
\[
R(X,Y)Z
=
\nabla_Y\nabla_XZ
-
\nabla_X\nabla_YZ
+
\nabla_{[X,Y]}Z.
\]
In local coordinates,
\[
R(\partial_i,\partial_j)\partial_k
=
R_{ijk}{}^l\partial_l,
\qquad
R_{ijkl}
=
\left\langle
R(\partial_i,\partial_j)\partial_k,
\partial_l
\right\rangle
=
g_{lm}R_{ijk}{}^m.
\]
Thus, for vector fields \( X=X^i\partial_i, \, Y=Y^j\partial_j, \, Z=Z^k\partial_k \), we have
\[
R(X,Y)Z
=
X^iY^jZ^kR_{ijk}{}^l\partial_l.
\]

The Gauss equation on \( \Sigma \) is
\[
R_{ijkl}
=
h_{ik}h_{jl}
-
h_{il}h_{jk}.
\]

With the above curvature convention, covariant derivatives satisfy
\[
(\nabla_i\nabla_j-\nabla_j\nabla_i)X^h
=
-R_{ijkl}g^{hl}X^k
=
-R_{ijk}{}^hX^k
\]
for a vector field \( X \), while for a covariant \( 1 \)-form \( \beta \),
\[
(\nabla_i\nabla_j-\nabla_j\nabla_i)\beta_k
=
R_{ijkm}g^{ml}\beta_l
=
R_{ijk}{}^l\beta_l.
\]

For a \( (1,2) \)-tensor \( T \), we write
\[
\nabla T
=
\{\nabla_lT^i{}_{jk}\},
\qquad
(\Delta T)^i{}_{jk}
=
g^{mn}\nabla_m\nabla_nT^i{}_{jk}.
\]

Finally, all norms on \( S^2 \) are taken with respect to \( \bar g \). For a tensor field \( T \), \( \|T\|_{L^2(S^2)}^2 := \int_{S^2}|T|_{\bar g}^2\,d\omega \), and for \( u\in H^2(S^2) \),
\[
\|u\|_{H^2(S^2)}^2
:=
\|u\|_{L^2(S^2)}^2
+
\|\bar\nabla u\|_{L^2(S^2)}^2
+
\|\bar\nabla^2u\|_{L^2(S^2)}^2.
\]

\section{Spectral Estimates for Higher Spherical Harmonic Modes}
In this section we collect the spectral and geometric estimates needed for the main coercivity argument. After the \( l=0 \) and \( l=1 \) spherical harmonic modes are removed, the operator \( \bar\nabla(\bar\Delta+2) \) has a spectral gap on \( S^2 \). We combine this fact with the expansions of the mean curvature and the trace-free second fundamental form for small radial graphs over the unit sphere.

Unless stated otherwise, let \( F:S^2\to\mathbb R^3, \, F(\omega)=(1+u(\omega))\omega \), where \( u\in C^3(S^2), \, \|u\|_{C^3(S^2)}\le\varepsilon, \, \varepsilon\in(0,\frac14). \) Set
\[
\Sigma:=F(S^2),
\qquad
\hat g:=F^*(g_{\mathbb R^3}|_{\Sigma}),
\qquad
d\hat\mu:=F^*(d\mu_{\Sigma}).
\]
We denote by \( \hat\nabla \) the Levi--Civita connection of \( \hat g \), and by \( \hat A=\{\hat h_{ij}\}:=F^*A_\Sigma \) the pull-back of the second fundamental form of \( \Sigma \). We denote by \( H_{\Sigma} \) the mean curvature of \( \Sigma \) with respect to the outward unit normal.

\subsection{Spherical harmonic estimates}
We begin with the estimate of linear part of \( |\bar\nabla H_{\Sigma}|_{\bar g}^2 \) and \( (k_{1,\Sigma}-k_{2,\Sigma})^2 \). We use the standard spherical harmonic decomposition of the Laplace--Beltrami operator on the round sphere; see, for example, \cite[Chapter II, \S5]{Chavel1984}.

Let \( \{Y_{l,m}\} \) be an \( L^2(S^2) \)-orthonormal basis of spherical harmonics, satisfying
\[
-\bar\Delta Y_{l,m}=\lambda_lY_{l,m},
\qquad
\lambda_l=l(l+1),
\]
and denote by \( \mathcal H_l \) the eigenspace of \( -\bar\Delta \) with eigenvalue \( \lambda_l \). Let \( P_l \) be the \( L^2(S^2) \)-orthogonal projection onto \( \mathcal H_l \). Define
\[
\mathscr Q(f)
:=
\int_{S^2}|\bar\nabla(\bar\Delta+2)f|_{\bar g}^2\,d\omega,
\qquad
\mathscr R(f)
:=
\int_{S^2}|(\bar\nabla^2f)^{\circ_{\bar g}}|_{\bar g}^2\,d\omega.
\]

\begin{prop}[Diagonalization of \( \mathscr Q \) and \( \mathscr R \)]
\label{lem:QR-diagonal}
Let \(u\in H^3(S^2)\), and write
\[
u=\sum_{l\ge0}u_l,
\qquad
u_l:=P_lu=\sum_{m=-l}^l u_{l,m}Y_{l,m},
\]
where
\[
u_{l,m}:=\langle u,Y_{l,m}\rangle_{L^2(S^2)}=\int_{S^2}u\,Y_{l,m}\,d\omega\in\mathbb R
\]
is the scalar Fourier coefficient of \(u\) with respect to the spherical harmonic \(Y_{l,m}\). Then
\[
\mathscr Q(u)=\sum_{l\ge0}\mathscr Q(u_l),
\qquad
\mathscr R(u)=\sum_{l\ge0}\mathscr R(u_l).
\]
Moreover,
\begin{align}
\mathscr Q(u)
&=
\sum_{l\ge0}\sum_{m=-l}^l \lambda_l(\lambda_l-2)^2|u_{l,m}|^2,
\label{eq:Q_spec_compressed}\\
\mathscr R(u)
&=
\sum_{l\ge0}\sum_{m=-l}^l \frac12\lambda_l(\lambda_l-2)|u_{l,m}|^2.
\label{eq:R_spec}
\end{align}
\end{prop}

\begin{proof}
We only compute \( \mathscr Q \). The argument for \( \mathscr R \) is similar.

Define
\[
u^{(N)}:=\sum_{l=0}^N u_l.
\]
Since
\[
\bar\Delta u_l=-\lambda_lu_l,
\]
we obtain
\[
\bar\nabla(\bar\Delta+2)u_l=(2-\lambda_l)\bar\nabla u_l.
\]

If \(l\neq l'\), then integration by parts and the \(L^2(S^2)\)-orthogonality of distinct spherical harmonic eigenspaces give
\[
\begin{aligned}
\int_{S^2}\left\langle \bar\nabla u_l,\bar\nabla u_{l'}\right\rangle_{\bar g}\,d\omega
&=
-\int_{S^2}u_l\,\bar\Delta u_{l'}\,d\omega\\
&=
\lambda_{l'}\int_{S^2}u_lu_{l'}\,d\omega\\
&=0.
\end{aligned}
\]
Therefore,
\[
\mathscr Q(u^{(N)})=\sum_{l=0}^N\mathscr Q(u_l).
\]

For each \(l\),
\[
\begin{aligned}
\mathscr Q(u_l)
&=
(\lambda_l-2)^2\int_{S^2}|\bar\nabla u_l|_{\bar g}^2\,d\omega\\
&=
\lambda_l(\lambda_l-2)^2\|u_l\|_{L^2(S^2)}^2.
\end{aligned}
\]
Since
\[
u_l=\sum_{m=-l}^l u_{l,m}Y_{l,m},
\]
and the \(Y_{l,m}\) are \(L^2(S^2)\)-orthonormal,
\[
\mathscr Q(u_l)=\sum_{m=-l}^l \lambda_l(\lambda_l-2)^2|u_{l,m}|^2.
\]

Since \(u\in H^3(S^2)\), we have
\[
u^{(N)}\to u
\qquad\text{in }H^3(S^2),
\]
and hence
\[
\bar\nabla(\bar\Delta+2)u^{(N)}\to \bar\nabla(\bar\Delta+2)u
\qquad\text{in }L^2(S^2).
\]
Thus
\[
\mathscr Q(u^{(N)})\to\mathscr Q(u).
\]
Letting \(N\to\infty\), we obtain
\[
\mathscr Q(u)=\sum_{l\ge0}\mathscr Q(u_l)=\sum_{l\ge0}\sum_{m=-l}^l \lambda_l(\lambda_l-2)^2|u_{l,m}|^2.
\]

The identities for \( \mathscr R \) follow similarly from the Bochner identity on \(S^2\).
\end{proof}

\begin{prop}[Trace-free Hessian spectral gap on \( S^2 \)]
\label{prop:tracefree_hessian_gap}
Let \( u\in H^2(S^2) \), and denote by \( u_{\ge2} \) its projection onto the
spherical harmonics of degrees \( l\ge2 \). Then, for \( k=0,1,2 \),
\begin{equation}
\label{eq:tracefree_hessian_gap}
\|\bar\nabla^ku_{\ge2}\|_{L^2(S^2)}^2
\le
C
\int_{S^2}
|(\bar\nabla^2u)^{\circ_{\bar g}}|_{\bar g}^2\,d\omega.
\end{equation}
\end{prop}

\begin{proof}
Writing
\(
u_{\ge2}
=\sum_{l\ge2}\sum_{m=-l}^l u_{l,m}Y_{l,m},
 \)
we have
\[
\|u_{\ge2}\|_{L^2}^2
=\sum_{l\ge2,m}|u_{l,m}|^2,
\qquad
\|\bar\nabla u_{\ge2}\|_{L^2}^2
=\sum_{l\ge2,m}\lambda_l|u_{l,m}|^2,
\]
and, by the Bochner identity,
\[
\|\bar\nabla^2u_{\ge2}\|_{L^2}^2
=\sum_{l\ge2,m}\lambda_l(\lambda_l-1)|u_{l,m}|^2.
\]
On the other hand, Proposition~\ref{lem:QR-diagonal} gives
\[
\int_{S^2}
|(\bar\nabla^2u)^{\circ_{\bar g}}|_{\bar g}^2\,d\omega
=\sum_{l\ge2,m}\frac12
\lambda_l(\lambda_l-2)|u_{l,m}|^2.
\]
Since \( \lambda_l\ge6 \) for \( l\ge2 \), the quantities \( 1 \), \( \lambda_l \), and \( \lambda_l(\lambda_l-1) \) are bounded by a universal multiple of \( \frac12\lambda_l(\lambda_l-2) \). This proves \eqref{eq:tracefree_hessian_gap}.
\end{proof}

\begin{lemma}[Elliptic spectral gap on \( S^2 \)]
\label{lem:elliptic_gap}
Let \( u\in H^3(S^2) \), and denote by \( u_{\ge2} \) its projection onto modes \( l\ge2 \). Then, for \( k=1,2,3 \),
\begin{equation}
\label{eq:elliptic_gap}
\|\bar\nabla^ku_{\ge2}\|_{L^2(S^2)}^2
\le
C
\|\bar\nabla(\bar\Delta+2)u\|_{L^2(S^2)}^2.
\end{equation}
\end{lemma}

\begin{proof}
By Proposition~\ref{lem:QR-diagonal},
\[
\|\bar\nabla(\bar\Delta+2)u\|_{L^2(S^2)}^2
=
\sum_{l\ge2,m}
\lambda_l(\lambda_l-2)^2|u_{l,m}|^2.
\]
By a direct computation using the spherical harmonic expansion (see also, e.g., \cite[Section~2.3, p.~3022]{DambrineLamboley2019}), we obtain

\[
\|u_{\ge2}\|_{H^3(S^2)}^2
\le
C
\sum_{l\ge2,m}(1+\lambda_l)^3|u_{l,m}|^2.
\]

Since \( \lambda_l\ge6 \) for \( l\ge2 \), there exists a constant \( C \) independent of \( l \) such that
\[
(1+\lambda_l)^3
\le
C\lambda_l(\lambda_l-2)^2.
\]
Therefore
\[
\|u_{\ge2}\|_{H^3(S^2)}^2
\le
C
\|\bar\nabla(\bar\Delta+2)u\|_{L^2(S^2)}^2,
\]
which implies \eqref{eq:elliptic_gap}.
\end{proof}

\subsection{Metric and mean curvature estimates}
The spectral estimates above are formulated on the fixed round sphere. We next transfer them to nearby surfaces by comparing the induced geometry of a small radial graph with the background metric and by identifying the linearized mean-curvature operator.

\begin{lemma}[Metric, measure, and gradient comparison]
\label{lem:pullback_L2_comparison}
Assume \( \|u\|_{C^3(S^2)}\le\varepsilon<\frac14. \) Then
\begin{equation}
\label{eq:metric_equiv}
\left(1-(2+2\varepsilon)\varepsilon\right)\bar g
\le
\hat g
\le
\left(1+(2+2\varepsilon)\varepsilon\right)\bar g,
\end{equation}
and
\begin{equation}
\label{eq:measure_equiv}
\left(1-(2+2\varepsilon)\varepsilon\right)d\omega
\le
d\hat\mu
\le
\left(1+(2+2\varepsilon)\varepsilon\right)d\omega.
\end{equation}
Moreover, for \( f\in C^1(S^2) \), viewed as a function on \( \Sigma \) by \( f_\Sigma=f\circ F^{-1} \),
\begin{equation}
\label{eq:L2_grad_comp}
\begin{aligned}
&\left(
1-
\frac{4\varepsilon+4\varepsilon^2}
{1+2\varepsilon+2\varepsilon^2}
\right)
\int_{S^2}|\bar\nabla f|_{\bar g}^2\,d\omega
\\
&\qquad\le
\int_\Sigma|\nabla_\Sigma f_\Sigma|_g^2\,d\mu_\Sigma
\\
&\qquad\le
\left(
1+
\frac{4\varepsilon+4\varepsilon^2}
{1-2\varepsilon-2\varepsilon^2}
\right)
\int_{S^2}|\bar\nabla f|_{\bar g}^2\,d\omega.
\end{aligned}
\end{equation}
\end{lemma}

\begin{proof}
Differentiating \( F(\omega)=(1+u(\omega))\omega \) gives
\[
\hat g_{ij}
=(1+u)^2\bar g_{ij}+\partial_iu \partial_ju.
\]
Hence, for \( X\in TS^2 \),
\[
(1-\varepsilon)^2\bar g(X,X)
\le
\hat g(X,X)
\le
\left((1+\varepsilon)^2+\varepsilon^2\right)\bar g(X,X),
\]
which implies \eqref{eq:metric_equiv}. The corresponding determinant estimate implies \eqref{eq:measure_equiv}.

Since \( \hat g=F^*g \), the map \( F:(S^2,\hat g)\rightarrow(\Sigma,g) \) is an isometry. Therefore
\[
\int_\Sigma|\nabla_\Sigma f_\Sigma|_g^2\,d\mu_\Sigma
=
\int_{S^2}|\hat\nabla f|_{\hat g}^2\,d\hat\mu.
\]
Taking inverses in \eqref{eq:metric_equiv}, we obtain
\[
\frac{1}{1+(2+2\varepsilon)\varepsilon}
|\bar\nabla f|_{\bar g}^2
\le
|\hat\nabla f|_{\hat g}^2
\le
\frac{1}{1-(2+2\varepsilon)\varepsilon}
|\bar\nabla f|_{\bar g}^2.
\]
It follows from \eqref{eq:measure_equiv} that
\[
\frac{1-(2+2\varepsilon)\varepsilon}
{1+(2+2\varepsilon)\varepsilon}
\int_{S^2}|\bar\nabla f|_{\bar g}^2\,d\omega
\le
\int_\Sigma|\nabla_\Sigma f_\Sigma|_g^2\,d\mu_\Sigma
\]
and
\[
\int_\Sigma|\nabla_\Sigma f_\Sigma|_g^2\,d\mu_\Sigma
\le
\frac{1+(2+2\varepsilon)\varepsilon}
{1-(2+2\varepsilon)\varepsilon}
\int_{S^2}|\bar\nabla f|_{\bar g}^2\,d\omega.
\]
Rewriting the two coefficients yields \eqref{eq:L2_grad_comp}.
\end{proof}

\begin{lemma}[Mean curvature expansion for small radial graphs]
\label{lem:mean.cur.expansion}
There exists a function \( Q(u) \) on \( S^2 \) such that
\begin{equation}
\label{eq:H_expand_Q}
H_\Sigma\circ F
=
2-(\bar\Delta+2)u+Q(u),
\end{equation}
where \( Q(0)=0, \, DQ(0)=0. \) Moreover, whenever \( \|u\|_{C^3(S^2)}\le\varepsilon \), one has
\begin{equation}
\label{eq:Q_grad_pt}
|\bar\nabla Q(u)|_{\bar g}
\le
C\varepsilon
\left(
|\bar\nabla^3u|_{\bar g}
+
|\bar\nabla^2u|_{\bar g}
+
|\bar\nabla u|_{\bar g}
\right).
\end{equation}
\end{lemma}

\begin{proof}
Define \( M(u):=H_{\Sigma}\circ F. \) The first and second fundamental forms of a radial graph depend smoothly on the \( 2 \)-jet \( (u,\bar\nabla u,\bar\nabla^2u) \), so \( M \) is a smooth second-order quasilinear operator near \( u=0 \). At the unit sphere, \( M(0)=2. \) For the variation \( F_t(\omega)=(1+tu(\omega))\omega \), the variational vector field at \( t=0 \) is \( u\nu \). The first variation formula for mean curvature (see, for example, \cite[Exercise~2.13]{Lee2019GeometricRelativity}) gives
\[
DM(0)u
=
-\bar\Delta u
-
\left(
|A|_{\bar g}^2
+
\operatorname{Ric}_{\mathbb R^3}(\nu,\nu)
\right)u
=
-(\bar\Delta+2)u,
\]
since \( |A|_{\bar g}^2=2 \) on the unit sphere and the ambient Ricci curvature vanishes.

Taylor's theorem gives \( M(u) = 2-(\bar\Delta+2)u+Q(u), \, Q(0)=DQ(0)=0. \) More explicitly, the exact radial-graph formulas for the first and second fundamental forms show that, pointwise on \( S^2 \), \( M(u) = \mathcal M \left( u,\bar\nabla u,\bar\nabla^2u \right) \), where \( \mathcal M \) is smooth in a fixed neighborhood of the zero jet; see, for example, the radial-graph formulas in \cite[Section~2, p.~4]{Cruz2017}. Thus \( Q(u) = \mathcal Q \left( u,\bar\nabla u,\bar\nabla^2u \right) \), where \( \mathcal Q(0,0,0)=0, \, D\mathcal Q(0,0,0)=0. \) Since \( |u| + |\bar\nabla u|_{\bar g} + |\bar\nabla^2u|_{\bar g} \le C\varepsilon \), the mean value theorem gives \( |D\mathcal Q| \le C\varepsilon \) along the segment joining the zero jet to \( (u,\bar\nabla u,\bar\nabla^2u) \). Applying the covariant chain rule therefore yields
\[
|\bar\nabla Q(u)|_{\bar g}
\le
C\varepsilon
\left(
|\bar\nabla^3u|_{\bar g}
+
|\bar\nabla^2u|_{\bar g}
+
|\bar\nabla u|_{\bar g}
\right),
\]
which proves \eqref{eq:Q_grad_pt}.
\end{proof}

\begin{prop}[Preliminary lower bound in the \( l\ge2 \) gauge]
\label{DH.lower.bound}
Assume
\[
\|u\|_{C^3(S^2)}\le\varepsilon<\frac14,
\qquad
u\perp_{L^2(S^2)}(\mathcal H_0\oplus\mathcal H_1).
\]
There exists a constant \( C>0 \), depending only on the constants in Lemmas~\ref{lem:elliptic_gap} and~\ref{lem:mean.cur.expansion}, such that
\begin{equation}
\label{eq:pre_tail}
\int_\Sigma
|\nabla_\Sigma H_\Sigma|_g^2
\,d\mu_\Sigma
\ge
\left[
1-
\frac{4\varepsilon+4\varepsilon^2}
{1+2\varepsilon+2\varepsilon^2}
-
C\varepsilon
\right]
\int_{S^2}
|\bar\nabla(\bar\Delta+2)u|_{\bar g}^2
\,d\omega.
\end{equation}

For every \( \delta'\in(0,1) \) there exists \( \varepsilon=\varepsilon(\bar g,\delta')\in(0,\frac14) \) such that
\[
\frac{4\varepsilon+4\varepsilon^2}
{1+2\varepsilon+2\varepsilon^2}
+
C\varepsilon
\le
\delta'.
\]
In particular, for
\[
\|u\|_{C^3(S^2)}\le\varepsilon(\bar g,\delta'),
\qquad
u\perp_{L^2(S^2)}(\mathcal H_0\oplus\mathcal H_1),
\]
we have
\[
\int_\Sigma
|\nabla_\Sigma H_\Sigma|_g^2
\,d\mu_\Sigma
\ge
(1-\delta')\int_{S^2}
|\bar\nabla(\bar\Delta+2)u|_{\bar g}^2
\,d\omega.
\]
\end{prop}

\begin{proof}
Applying Lemma~\ref{lem:pullback_L2_comparison} to \( f=H_\Sigma\circ F \) gives
\begin{equation}
\label{eq:pullback_lower}
\int_\Sigma
|\nabla_\Sigma H_\Sigma|_g^2
\,d\mu_\Sigma
\ge
\left(
1-
\frac{4\varepsilon+4\varepsilon^2}
{1+2\varepsilon+2\varepsilon^2}
\right)
\int_{S^2}
|\bar\nabla(H_\Sigma\circ F)|_{\bar g}^2
\,d\omega.
\end{equation}
By Lemma~\ref{lem:mean.cur.expansion},
\[
\bar\nabla(H_\Sigma\circ F)
=
-\bar\nabla(\bar\Delta+2)u
+
\bar\nabla Q(u).
\]
For \( s\in(0,1) \),
\[
|a+b|^2
\ge
(1-s)|a|^2-s^{-1}|b|^2.
\]
Hence
\[
\int_{S^2}
|\bar\nabla(H_\Sigma\circ F)|_{\bar g}^2
\,d\omega
\ge
(1-s)\int_{S^2}
|\bar\nabla(\bar\Delta+2)u|_{\bar g}^2
\,d\omega
-
s^{-1}
\int_{S^2}|\bar\nabla Q(u)|^2\,d\omega.
\]
Since \( u=u_{\ge2} \), Lemma~\ref{lem:mean.cur.expansion} and Lemma~\ref{lem:elliptic_gap} imply
\[
\int_{S^2}|\bar\nabla Q(u)|^2\,d\omega
\le
C\varepsilon^2\int_{S^2}
|\bar\nabla(\bar\Delta+2)u|_{\bar g}^2
\,d\omega.
\]

Taking \( s=\varepsilon \), we obtain
\[
\begin{aligned}
\int_{S^2}
|\bar\nabla(H_\Sigma\circ F)|_{\bar g}^2
\,d\omega
&\ge
(1-\varepsilon)\int_{S^2}
|\bar\nabla(\bar\Delta+2)u|_{\bar g}^2
\,d\omega
-
C\varepsilon \int_{S^2}
|\bar\nabla(\bar\Delta+2)u|_{\bar g}^2
\,d\omega\\
&\ge
\bigl(1-C\varepsilon\bigr)\int_{S^2}
|\bar\nabla(\bar\Delta+2)u|_{\bar g}^2
\,d\omega,
\end{aligned}
\]
after enlarging \( C \) if necessary. Substituting this into \eqref{eq:pullback_lower}, we get
\[
\begin{aligned}
\int_\Sigma
|\nabla_\Sigma H_\Sigma|_g^2
\,d\mu_\Sigma
&\ge
\left(
1-
\frac{4\varepsilon+4\varepsilon^2}
{1+2\varepsilon+2\varepsilon^2}
\right)
\bigl(1-C\varepsilon\bigr)\int_{S^2}
|\bar\nabla(\bar\Delta+2)u|_{\bar g}^2
\,d\omega\\
&\ge
\left[
1-
\frac{4\varepsilon+4\varepsilon^2}
{1+2\varepsilon+2\varepsilon^2}
-
C\varepsilon
\right]\int_{S^2}
|\bar\nabla(\bar\Delta+2)u|_{\bar g}^2
\,d\omega,
\end{aligned}
\]
which proves \eqref{eq:pre_tail}.

\end{proof}

\subsection{Estimate of the trace-free second fundamental form}
The preceding estimates control the gradient of mean curvature in the evolution equation for Willmore energy. To compare it with the positive non-umbilical term, we also need to estimate \( (k_1-k_2)^2=2|A^\circ|^2. \)

\begin{prop}[Quantitative expansion of \( A^\circ \) for a radial graph]
\label{prop:tracefree.reminder.estimate}
For sufficiently small \( \varepsilon\in(0,\frac14) \), whenever \( \|u\|_{C^3(S^2)}\le\varepsilon \), there exists a constant \( C>0 \) such that
\begin{equation}
\label{eq:A0_explicit}
\hat A^{\circ_{\hat g}}
=-(\bar\nabla^2u)^{\circ_{\bar g}}+\mathcal R,
\end{equation}
and
\begin{equation}
\label{eq:A0_remainder_compressed}
|\mathcal R|_{\bar g}
\le
C
\left(
|u|\,|\bar\nabla^2u|_{\bar g}
+|\bar\nabla u|_{\bar g}^2
+u^2
\right).
\end{equation}
The same estimate holds with \( \hat A^{\circ_{\hat g}} \) replaced by \( \hat A^{\circ_{\bar g}} \).
\end{prop}

\begin{proof}
Set \( r:=1+u \), \( r_i:=\bar\nabla_ir \), \( r_{ij}:=\bar\nabla_i\bar\nabla_jr \), and \( r^i:=\bar g^{ij}r_j \). Set \( u_i:=\bar\nabla_iu \), and \( u_{ij}:=\bar\nabla_i\bar\nabla_ju \). For a local \( \bar g \)-orthonormal frame \( \{e_i\}_{i=1,2} \) on \( S^2 \), \( F_i=r_i\omega+re_i \), so
\begin{equation}
\label{eq:g_formula}
\hat g_{ij}=r^2\bar g_{ij}+r_i r_j.
\end{equation}
Notice that the vector \( \mathcal N=r\omega-\bar\nabla r \) is normal to the graph. Define \( W:=|\mathcal N| =\sqrt{r^2+|\bar\nabla r|_{\bar g}^2}. \) Then the outward unit normal is \( \hat\nu=\mathcal N/W \). Since \( r=1+u \) and \( \|u\|_{C^1(S^2)}\le\varepsilon \), we have
\begin{equation}
\label{eq:W_expand}
W
=1+u+O\left(u^2+|\bar\nabla u|_{\bar g}^2\right),
\qquad
\frac1W
=1-u+O\left(u^2+|\bar\nabla u|_{\bar g}^2\right).
\end{equation}

Using \( D_{e_i}e_j =\bar\nabla_{e_i}e_j-\bar g_{ij}\omega \) and the convention \( \hat h_{ij}=-(D_{e_i}F_j,\hat\nu) \), a direct computation gives
\[
\hat h_{ij}
=
\frac1W
\left(
r^2\bar g_{ij}+2r_ir_j-rr_{ij}
\right).
\]
Substituting \( r=1+u \) and using \eqref{eq:W_expand}, we obtain
\begin{equation}
\label{eq:h_expand}
\hat h_{ij}
=
\bar g_{ij}-u_{ij}+u\bar g_{ij}+\mathtt{Q}_{ij},
\end{equation}
where
\begin{equation}
\label{eq:Q_tensor_compressed}
|\mathtt{Q}|_{\bar g}
\le
C
\left(
|u|\,|\bar\nabla^2u|_{\bar g}
+|\bar\nabla u|_{\bar g}^2
+u^2
\right).
\end{equation}
Taking the trace-free part with respect to \( \bar g \), the pure-trace terms \( \bar g \) and \( u\bar g \) disappear. Hence
\begin{equation}
\label{eq:h_tf_bg}
\hat A^{\circ_{\bar g}}
=-(\bar\nabla^2u)^{\circ_{\bar g}}+\mathcal R,
\end{equation}
where \( \mathcal R \) satisfies the same bound as \( \mathtt{Q} \) in \eqref{eq:Q_tensor_compressed}.

It remains to compare the trace-free projections associated with \( \bar g \) and \( \hat g \). From \eqref{eq:g_formula}, \( \hat g = \bar g+(2u+u^2)\bar g+du\otimes du \), and therefore
\begin{equation}
\label{eq:metric_inverse_short}
|\hat g-\bar g|_{\bar g}
+|\hat g^{-1}-\bar g^{-1}|_{\bar g}
\le
C
\left(|u|+|\bar\nabla u|_{\bar g}^2\right).
\end{equation}
Moreover, the explicit inverse
\[
\hat g^{ij}
=
\frac1{r^2}
\left(
\bar g^{ij}
-
\frac{r^ir^j}{r^2+|\bar\nabla r|_{\bar g}^2}
\right)
\]
shows that
\begin{equation}
\label{eq:pure_trace_short}
|((1+u)\bar g)^{\circ_{\hat g}}|_{\bar g}
\le
C|\bar\nabla u|_{\bar g}^2.
\end{equation}

Rewriting \eqref{eq:h_expand}, we have \( \hat A=(1+u)\bar g+T, \, T:=-\bar\nabla^2u+\mathtt{Q}. \) Since \( \bar g_{ij}+u\,\bar g_{ij} \) is pure trace with respect to \( \bar g \), we have \( \hat h^{\circ_{\bar g}}=T^{\circ_{\bar g}}. \) Moreover,
\begin{equation}
\label{T.estimate}
|T|_{\bar g}
\le
\left(|\bar\nabla^2u|_{\bar g}+|\mathtt{Q}|_{\bar g}\right)
\le
C
\left(|\bar\nabla^2u|_{\bar g}+|u|\,|\bar\nabla^2u|_{\bar g}+|\bar\nabla u|_{\bar g}^2+u^2\right).
\end{equation}
Now
\[
\hat A^{\circ_{\hat g}}=\hat h^{\circ_{\hat g}}
=
T^{\circ_{\hat g}}
+
(\bar g+u\bar g)^{\circ_{\hat g}}.
\]

Using the definition of the trace-free projection,
\[
T^{\circ_{\hat g}}-T^{\circ_{\bar g}}
=
-\frac12
\left(
\operatorname{tr}_{\hat g}T-
\operatorname{tr}_{\bar g}T
\right)\hat g
-
\frac12
(\operatorname{tr}_{\bar g}T)(\hat g-\bar g).
\]
Together with \eqref{eq:metric_inverse_short}, this yields
\[
|T^{\circ_{\hat g}}-T^{\circ_{\bar g}}|_{\bar g}
\le
C
\left(|u|+|\bar\nabla u|_{\bar g}^2\right)|T|_{\bar g}.
\]
By the smallness of \( \|u\|_{C^3(S^2)} \), combining this estimate with \eqref{eq:Q_tensor_compressed}, \eqref{eq:pure_trace_short} and \eqref{T.estimate} gives
\[
|\hat A^{\circ_{\hat g}}-\hat A^{\circ_{\bar g}}|_{\bar g}
\le
C
\left(
|u|\,|\bar\nabla^2u|_{\bar g}
+|\bar\nabla u|_{\bar g}^2
+u^2
\right).
\]
Combining this with \eqref{eq:h_tf_bg} proves \eqref{eq:A0_explicit}--\eqref{eq:A0_remainder_compressed}.
\end{proof}

The estimates obtained so far become coercive only after the kernel generated by the \( l=0 \) and \( l=1 \) spherical harmonics is removed. Since these modes come from the geometric symmetries of dilation and translation, we now impose a nonlinear transformation adapted to exactly those symmetries.

\section{Gauge Fixing of the Low Spherical Harmonic Modes}
In this section we eliminate the \( l=0 \) and \( l=1 \) spherical harmonic modes of a small radial graph over the unit sphere by a suitable translation and scaling. This transformation allows us to apply the spectral estimates from the previous section. Geometrically, the construction below chooses the center and scale of the nearby sphere so that the new graph function lies entirely in the \( l\ge2 \) subspace.

Throughout, \( S^2 \) is equipped with its standard metric \( \bar g \) and area measure \( d\omega \), and \( \mathcal H_l \) denotes the eigenspace of \( -\bar\Delta \) with eigenvalue \( l(l+1) \).

\begin{prop}[Gauge fixing of the \( l=0,1 \) modes by translation and scaling]
\label{lem:gauge_fix_l01_strict}
There exist constants \( \varepsilon_0>0 \) and \( C>0 \) with the following property. Let \( \Sigma = \{(1+u(\omega))\omega:\omega\in S^2\} \subset\mathbb R^3 \) be a smooth closed strictly convex surface satisfying \( \|u\|_{C^3(S^2)}\le\varepsilon_0. \) Then there exist \( a\in\mathbb R^3 \) and \( \rho>0 \) such that \( \Sigma^{a,\rho}:=\rho(\Sigma-a) \) can be written as \( \Sigma^{a,\rho} = \{(1+\check u(\omega))\omega:\omega\in S^2\} \), where
\begin{equation}
\label{eq:l01_gauge_exact}
\int_{S^2}\check u\,d\omega=0,
\qquad
\int_{S^2}\check u(\omega)\omega\,d\omega=0.
\end{equation}
Equivalently, \( \check u\perp_{L^2(S^2)}(\mathcal H_0\oplus\mathcal H_1). \) Moreover,
\begin{equation}
\label{eq:a_rho_size}
|a|+|\rho-1|
\le
C\|u\|_{C^1(S^2)},
\end{equation}
and
\begin{equation}
\label{eq:checku_C3_bound}
\|\check u\|_{C^3(S^2)}
\le
C\|u\|_{C^3(S^2)}.
\end{equation}
\end{prop}

\begin{proof}
For \( (u,a,\rho) \) near \( (0,0,1) \), set
\[
Z_{u,a,\rho}(\vartheta)
:=
\rho\bigl((1+u(\vartheta))\vartheta-a\bigr).
\]
Choose \( \varepsilon_1>0 \) sufficiently small such that, whenever \( \|u\|_{C^3(S^2)}+|a|+|\rho-1| < \varepsilon_1 \),, we have \( |Z_{u,a,\rho}(\vartheta)|\ge\frac12 \,\text{for all }\vartheta\in S^2. \) Define
\[
\Phi_{u,a,\rho}(\vartheta)
:=
\frac{Z_{u,a,\rho}(\vartheta)}
{|Z_{u,a,\rho}(\vartheta)|}, \qquad \text{for}\, \|u\|_{C^3(S^2)}+|a|+|\rho-1|<\varepsilon_1.
\]
Since the radial projection \( P(z)=\frac{z}{|z|} \) is smooth away from the origin, \( \Phi_{u,a,\rho}=P\circ Z_{u,a,\rho} \) is \( C^3 \) and depends continuously on \( (u,a,\rho) \) in the \( C^3 \)-topology. Moreover, \( \Phi_{0,0,1}=\operatorname{Id}_{S^2}. \) By decreasing \(\varepsilon_1 \) if necessary, we may assume that \( \Phi_{u,a,\rho} \) is a local diffeomorphism whenever \( \|u\|_{C^3(S^2)}+|a|+|\rho-1| < \varepsilon_1. \) Since \( \Phi_{u,a,\rho} \) is homotopic to the identity, it has degree one. Being a local diffeomorphism of the compact connected manifold \( S^2 \), it is a covering map; degree one then implies that it is one-to-one, and hence a global \( C^3 \)-diffeomorphism of \( S^2 \); see, for example, \cite[Lemma~3.3, p.~9]{GoldsteinGrochulskaHajlasz2025}.

Consequently, \( \rho(\Sigma-a) \) has a unique radial representation \( \rho(\Sigma-a) = \{(1+\check u^{\,u,a,\rho}(\omega))\omega:\omega\in S^2\} \), and
\begin{equation}
\label{check.u.expression}
\check u^{\,u,a,\rho}
\bigl(\Phi_{u,a,\rho}(\vartheta)\bigr)
=
|Z_{u,a,\rho}(\vartheta)|-1.
\end{equation}

Let \( J_{u,a,\rho} \) be the Jacobian of \( \Phi_{u,a,\rho} \), defined by \( \Phi_{u,a,\rho}^*d\omega = J_{u,a,\rho}\,d\omega. \) By decreasing \( \varepsilon_1 \) once more if necessary, the maps
\[
(u,a,\rho)\longmapsto Z_{u,a,\rho},
\qquad
(u,a,\rho)\longmapsto\Phi_{u,a,\rho},
\qquad
(u,a,\rho)\longmapsto J_{u,a,\rho}
\]
are \( C^1 \) whenever \( \|u\|_{C^3(S^2)}+|a|+|\rho-1| < \varepsilon_1. \) Define
\[
\Psi(u,a,\rho)
:=
\left(
\int_{S^2}\check u^{\,u,a,\rho}(\omega)\,d\omega,
\int_{S^2}\check u^{\,u,a,\rho}(\omega)\omega\,d\omega
\right).
\]
Using \eqref{check.u.expression} and changing variables \( \omega=\Phi_{u,a,\rho}(\vartheta) \), we obtain
\[
\Psi(u,a,\rho)
=
\left(
\int_{S^2}(|Z(\vartheta)|-1)J(\vartheta)\,d\omega,
\int_{S^2}(|Z(\vartheta)|-1)\Phi(\vartheta) J(\vartheta)\,d\omega
\right),
\]
where, for brevity, \( Z=Z_{u,a,\rho} \), \( \Phi=\Phi_{u,a,\rho} \), and \( J=J_{u,a,\rho} \). Hence \( \Psi \) is \( C^1 \). We next show that the linearization of \( \Psi \) with respect to \( (a,\rho) \) at \( (0,0,1) \) is invertible.

At the unit sphere \( u=0 \), a variation of the scaling factor \( 1+s\in\mathbb{R} \) gives
\[
\left.
\frac{d}{ds}
\right|_{s=0}
\check u^{\,0,0,1+s}(\omega)
=1,
\]
while a translation in the direction \( b\in\mathbb R^3 \) gives
\[
\left.
\frac{d}{ds}
\right|_{s=0}
\check u^{\,0,sb,1}(\omega)
=-b\cdot\omega.
\]
Recall that
\[
\mathcal H_0=\operatorname{span}\{1\},
\qquad
\mathcal H_1=\operatorname{span}\{\omega_1,\omega_2,\omega_3\},
\]
where \( \omega=(\omega_1,\omega_2,\omega_3)\in S^2\subset\mathbb R^3. \) Thus the moment conditions in \eqref{eq:l01_gauge_exact} are precisely equivalent to \( \check u\perp_{L^2(S^2)}(\mathcal H_0\oplus\mathcal H_1). \) Moreover, using the standard symmetry identities
\[
\int_{S^2}\omega\,d\omega=0,
\qquad
\int_{S^2}\omega_i\omega_j\,d\omega
=
\frac{4\pi}{3}\delta_{ij},
\]
we obtain
\[
\begin{aligned}
D_{(a,\rho)}\Psi(0,0,1)[b,\eta]
&=
\left(
\int_{S^2}(\eta-b\cdot\omega)\,d\omega,
\int_{S^2}(\eta-b\cdot\omega)\omega\,d\omega
\right)\\
&=
\left(
4\pi\eta,
-\frac{4\pi}{3}b
\right),
\end{aligned}
\]
which is an invertible linear map from \( \mathbb R^3\times\mathbb R \) onto \( \mathbb R\times\mathbb R^3 \).

The Banach space implicit function theorem \cite[Theorem~4.E, pp.~250--251]{Zeidler1995Main} therefore yields that there exist an open neighborhood \( U_0\subset C^3(S^2) \) of \( 0 \), an open neighborhood \( W_0\subset\mathbb R^3\times\mathbb R \) of \( (0,1) \), and a unique \( C^1 \) map \( \gamma:U_0\longrightarrow W_0, \, \gamma(u)=(a(u),\rho(u)) \), such that \( \gamma(0)=(0,1) \) and
\begin{equation}\label{implicit.function.eq}
\Psi(u,a(u),\rho(u))=(0,0)
\qquad
\text{for every }u\in U_0.
\end{equation}
Moreover, if \( u\in U_0, \, (a,\rho)\in W_0 \), and \( \Psi(u,a,\rho)=(0,0) \), then \( (a,\rho)=(a(u),\rho(u)). \) Thus the corresponding radial function \( \check u \) satisfies \eqref{eq:l01_gauge_exact}.

We next record the quantitative bounds. Let \( x_0:=(0,1), \, x_*:=(a(u),\rho(u))\in \mathbb R^3\times\mathbb R \), where \( x_* \) is the solution of \eqref{implicit.function.eq}. At \( (u,0,1) \), one has \( \check u^{\,u,0,1}=u \), and hence
\begin{equation}
\label{Psi.estimate}
|\Psi(u,x_0)|
\le
C\|u\|_{C^0(S^2)}
\le
C\|u\|_{C^1(S^2)}.
\end{equation}
Since \( \Psi \) is \( C^1 \), the map \( (u,x)\longmapsto D_x\Psi(u,x) \) is continuous in a neighborhood of \( (0,x_0) \). Let
\[
\mathcal X:=C^3(S^2),
\qquad
\mathcal Y:=\mathbb R^3\times\mathbb R ,
\qquad
A_0:=D_x\Psi(0,x_0).
\]
Since \( A_0 \) is invertible and the map \( u\longmapsto D_x\Psi(u,x_0) \) is continuous at \( u=0 \), the openness of the set of invertible linear maps and the continuity of the inversion map imply that there exist an open neighborhood \( U_1\subset \mathcal X \) of \( 0 \) and a constant \( M>0 \) such that, for every \( u\in U_1 \), \( D_x\Psi(u,x_0) \) is invertible and \( \|D_x\Psi(u,x_0)^{-1}\|\le M. \)

We now keep this constant \( M \) fixed. Since the map \( (u,x)\longmapsto D_x\Psi(u,x) \) is continuous at \( (0,x_0)\in \mathcal X\times\mathcal Y \), there exist an open neighborhood \( U_2 \times W \subset \mathcal X \times \mathcal Y \) of \( (0,x_0) \) such that \( \|D_x\Psi(u,x)-A_0\| \le \frac{1}{4M} \) for every \( u\in U_2 \) and every \( x\in W \).  In particular, we have \( \|D_x\Psi(u,x_0)-A_0\| \le \frac{1}{4M} \) for every \( u\in U_2 \). Consequently, for every \( u\in U_2 \) and every \( x\in W \), the triangle inequality gives
\[
\begin{aligned}
&\|D_x\Psi(u,x)-D_x\Psi(u,x_0)\|\\
&\le
\|D_x\Psi(u,x)-A_0\|
+
\|D_x\Psi(u,x_0)-A_0\|\\
&\le
\frac{1}{4M}+\frac{1}{4M}\\
&=
\frac{1}{2M}.
\end{aligned}
\]

Since the solution map \( \gamma:U_0\longrightarrow W_0, \, \gamma(u)=(a(u),\rho(u)) \), is continuous and satisfies \( \gamma(0)=x_0 \), and since \( W \) is an open neighborhood of \( x_0 \), there exists an open neighborhood \( U_3\subset U_0 \) of \( 0 \) such that \( \gamma(U_3)\subset W. \)

Finally, set \( U:=U_1\cap U_2\cap U_3. \) Since \( U_1 \), \( U_2 \), and \( U_3 \) are open neighborhoods of \( 0 \) in \( C^3(S^2) \), their intersection \( U \) is also an open neighborhood of \( 0 \). Hence, choose \( 0<\varepsilon_0<\varepsilon_1 \) such that
\[
\mathcal B_{\varepsilon_0}
:=
\left\{
u\in C^3(S^2):
\|u\|_{C^3(S^2)}\le\varepsilon_0
\right\}
\subset U.
\]
Thus, for every \( u\in\mathcal B_{\varepsilon_0} \), \( D_x\Psi(u,x_0)\text{ is invertible}, \, \|D_x\Psi(u,x_0)^{-1}\|\le M \), and \( \|D_x\Psi(u,x)-D_x\Psi(u,x_0)\| \le \frac{1}{2M} \, \text{for every }x\in W. \) Moreover, since \( u\in U_3 \), the solution \( x_*=\gamma(u) \) belongs to \( W \). Here both \( M \) and \( W \) are independent of the particular choice of \( u\in\mathcal B_{\varepsilon_0} \). Set \( h:=x_*-x_0=(a,\rho-1). \) Since \( x_*\in W \) and \( W \) has been chosen to be a ball centered at \( x_0 \), the segment \( x_0+th,\, 0\le t\le1 \), is contained in \( W \). By the fundamental theorem of calculus,
\[
0=\Psi(u,x_*)
=
\Psi(u,x_0)+D_x\Psi(u,x_0)h+R(h),
\]
where
\[
R(h)
=
\int_0^1
\left(D_x\Psi(u,x_0+th)-D_x\Psi(u,x_0)\right)h\,dt.
\]
Since the segment \( x_0+th \) lies in the chosen neighborhood of \( x_0 \), we have
\[
|R(h)|
\le
\int_0^1
\|D_x\Psi(u,x_0+th)-D_x\Psi(u,x_0)\|\,|h|\,dt
\le
\frac{1}{2M}|h|.
\]

From
\[
0=
\Psi(u,x_0)+D_x\Psi(u,x_0)h+R(h),
\]
we obtain
\[
D_x\Psi(u,x_0)h
=
-\Psi(u,x_0)-R(h).
\]
Since \( D_x\Psi(u,x_0) \) is invertible, applying \( D_x\Psi(u,x_0)^{-1} \) gives
\[
h
=
-D_x\Psi(u,x_0)^{-1}\Psi(u,x_0)
-
D_x\Psi(u,x_0)^{-1}R(h).
\]
Taking norms and using \( \|D_x\Psi(u,x_0)^{-1}\|\le M \), we get
\[
|h|
\le
M|\Psi(u,x_0)|+M|R(h)|.
\]
Using the estimate for \( R(h) \), this implies
\[
|h|
\le
M|\Psi(u,x_0)|
+
M\cdot \frac{1}{2M}|h|
=
M|\Psi(u,x_0)|+\frac12|h|.
\]
Absorbing the last term into the left-hand side yields \( |h|\le 2M|\Psi(u,x_0)|. \) Since \( h=(a,\rho-1) \), we have \( |a|+|\rho-1|\le C|h|. \) Combining this with \eqref{Psi.estimate}, it holds that \( |a|+|\rho-1| \le C|h| \le C\|u\|_{C^1(S^2)}. \) This proves \eqref{eq:a_rho_size}.

It remains to prove the \( C^3 \) estimate for \( \check u \). By decreasing \( \varepsilon_0 \) if necessary, we may also assume that whenever \( \|u\|_{C^3(S^2)}+|a|+|\rho-1| < \varepsilon_0 \), the inverse diffeomorphism \( \Phi_{u,a,\rho}^{-1} \) has uniformly bounded \( C^3 \)-norm. Setting \( f(\vartheta):=|Z_{u,a,\rho}(\vartheta)|-1 \), we have, by \eqref{check.u.expression}, \( \check u=f\circ\Phi_{u,a,\rho}^{-1}. \) It follows that
\[
\|\check u\|_{C^3(S^2)}
\le
C\|f\|_{C^3(S^2)}
\le
C\bigl(
\|u\|_{C^3(S^2)}+|a|+|\rho-1|
\bigr).
\]
Using \eqref{eq:a_rho_size}, we conclude that
\[
\|\check u\|_{C^3(S^2)}
\le
C\|u\|_{C^3(S^2)},
\]
which proves \eqref{eq:checku_C3_bound}.
\end{proof}

We have therefore reduced every sufficiently small radial graph, up to translation and scaling, to the \( l\ge2 \) gauge while retaining quantitative \( C^3 \)-control. The spectral estimates and radial-graph expansions can now be assembled into the geometric inequality needed for the flow.

\section{The Main Coercive Estimate and Willmore Monotonicity}
The lemmas collected in this section provide the technical preparation for the proof of the main theorem in the next section. We first prove the coercive inequality on small graphs defined on \( S^2 \) whose parametrized function is orthogonal to the eigenspace \( \mathcal H_0\oplus\mathcal H_1 \). Then we record the transformation laws needed to transfer the estimate between different centers and scales.

\begin{lemma}[Coercivity in the \( l\ge2 \) gauge]
\label{lem:l-ge-2-coercivity}
For every \( \delta>0 \), there exists \( \varepsilon=\varepsilon(\bar g,\delta)\in\left(0,\frac14\right) \) such that the following holds. Let \( \Gamma = \{(1+v(\omega))\omega:\omega\in S^2\} \) be a smooth strictly convex radial graph satisfying
\[
\|v\|_{C^3(S^2)}\le\varepsilon,
\qquad
v\perp_{L^2(S^2)}(\mathcal H_0\oplus\mathcal H_1).
\]
Then
\begin{equation}
\label{key.inequality}
\int_\Gamma
|\nabla_\Gamma H_\Gamma|_{g_\Gamma}^2\,d\mu_\Gamma
\ge
(4-\delta)
\int_\Gamma
(k_{1,\Gamma}-k_{2,\Gamma})^2\,d\mu_\Gamma.
\end{equation}
\end{lemma}

\begin{proof}
If \( \delta\ge4 \), the conclusion is immediate, so we assume that \( 0<\delta<4 \).

By the preceding definition, we write
\[
\mathscr Q(v)
:=
\int_{S^2}|\bar\nabla(\bar\Delta+2)v|_{\bar g}^2\,d\omega,
\qquad
\mathscr R(v)
:=
\int_{S^2}|(\bar\nabla^2v)^{\circ_{\bar g}}|_{\bar g}^2\,d\omega.
\]
Since \( v \) contains only spherical harmonic modes of degree \( l\ge2 \), Proposition~\ref{lem:QR-diagonal} gives
\[
\mathscr Q(v)
=
\sum_{l\ge2,m}
\lambda_l(\lambda_l-2)^2|v_{l,m}|^2
\]
and
\[
\mathscr R(v)
=
\sum_{l\ge2,m}\frac12
\lambda_l(\lambda_l-2)|v_{l,m}|^2.
\]
Because \( \lambda_l=l(l+1)\ge6 \) for \( l\ge2 \),
\begin{equation}
\label{eq:l-ge-2-spectral-gap}
\mathscr Q(v)\ge8\mathscr R(v).
\end{equation}

Let \( \eta\in(0,1) \) be chosen later. By Proposition~\ref{DH.lower.bound}, there exists \( \varepsilon=\varepsilon(\bar g,\eta)\in(0,\frac14) \) such that, whenever \( \|v\|_{C^3(S^2)}\le \varepsilon(\bar g,\eta) \), it holds that
\begin{equation}
\label{eq:gradH-from-E}
\int_\Gamma
|\nabla_\Gamma H_\Gamma|_{g_\Gamma}^2\,d\mu_\Gamma
\ge
(1-\eta)\mathscr Q(v).
\end{equation}

On the other hand, Proposition~\ref{prop:tracefree.reminder.estimate} gives that \( \hat A_\Gamma^{\circ_{\hat g}} = -(\bar\nabla^2v)^{\circ_{\bar g}}+\mathcal R \), with
\[
|\mathcal R|_{\bar g}
\le
C
\left(
|v|\,|\bar\nabla^2v|_{\bar g}
+
|\bar\nabla v|_{\bar g}^2
+
v^2
\right).
\]
Then for \( \|v\|_{C^3(S^2)}\le\varepsilon(\bar g,\eta) \), using Proposition~\ref{prop:tracefree_hessian_gap}, we have
\begin{equation}
\label{trace-free.2-form.estimate}
\int_{S^2}
|\mathcal R|_{\bar g}^2\,d\omega
\le
C\varepsilon(\bar g,\eta)^2
\sum_{k=0}^{2}
\|\bar\nabla^kv\|_{L^2(S^2)}^2
\le
C\varepsilon(\bar g,\eta)^2\mathscr R(v).
\end{equation}
Moreover, Lemma~\ref{lem:pullback_L2_comparison} implies that
\begin{equation}
\label{2-tensor.comparison}
|T|_{g_\Gamma}^2\,d\mu_\Gamma
\le
(1+C\varepsilon(\bar g,\eta))
|T|_{\bar g}^2\,d\omega
\end{equation}
for every covariant \( 2 \)-tensor \( T \).

Using \( \hat A_\Gamma^{\circ_{\hat g}} = -(\bar\nabla^2v)^{\circ_{\bar g}}+\mathcal R \) and Young's inequality
\[
|X+Y|^2
\le
(1+\varepsilon(\bar g,\eta))|X|^2
+
C\varepsilon(\bar g,\eta)^{-1}|Y|^2,
\]
together with \eqref{trace-free.2-form.estimate} and \eqref{2-tensor.comparison} we therefore obtain
\[
\int_\Gamma
|A_\Gamma^\circ|_{g_\Gamma}^2\,d\mu_\Gamma
\le
(1+C\varepsilon(\bar g,\eta))\mathscr R(v).
\]
Since \( (k_{1,\Gamma}-k_{2,\Gamma})^2 = 2|A_\Gamma^\circ|_{g_\Gamma}^2 \), we have
\begin{equation}
\label{eq:k1k2-upper-by-hessian}
\int_\Gamma
(k_{1,\Gamma}-k_{2,\Gamma})^2\,d\mu_\Gamma
\le
2(1+C\varepsilon(\bar g,\eta))\mathscr R(v).
\end{equation}
Combining \eqref{eq:l-ge-2-spectral-gap}, \eqref{eq:gradH-from-E}, and \eqref{eq:k1k2-upper-by-hessian}, we find
\[
\int_\Gamma
|\nabla_\Gamma H_\Gamma|_{g_\Gamma}^2\,d\mu_\Gamma
\ge
\frac{4(1-\eta)}
{1+C\varepsilon(\bar g,\eta)}
\int_\Gamma
(k_{1,\Gamma}-k_{2,\Gamma})^2\,d\mu_\Gamma.
\]

Let \( \eta > 0 \) be sufficiently small so that
\begin{align*}
    4\eta + (4 - \delta)C\varepsilon(\bar{g}, \eta) \leq \delta.
\end{align*}
This is possible since \( \varepsilon(\bar{g}, \eta) \to 0 \) as \( \eta \to 0 \). Since \( \eta \) depends on \( \delta \), so does \( \varepsilon(\bar g,\eta) \). Consequently, for \( \|v\|_{C^3(S^2)}\le\varepsilon(\bar g,\delta) \), it holds that
\[
\int_\Gamma
|\nabla_\Gamma H_\Gamma|_{g_\Gamma}^2\,d\mu_\Gamma
\ge
(4 - \delta)
\int_\Gamma
(k_{1,\Gamma}-k_{2,\Gamma})^2\,d\mu_\Gamma,
\]
which is exactly \eqref{key.inequality}.
\end{proof}

\begin{lemma}[Transformation of the coercive inequality under translation and scaling]
\label{lem:coercive-inequality-under-rigid-scaling}
Let \( \Sigma\subset\mathbb R^3 \) be a smooth closed strictly convex surface and, for \( b\in\mathbb R^3 \) and \( \lambda>0 \), set \( \Sigma^{b,\lambda}:=\lambda(\Sigma-b). \) We denote the induced metric, Levi--Civita connection, mean curvature, principal curvatures, and area element of \( \Sigma^{b,\lambda} \) by
\[
g^{b,\lambda},
\qquad
\nabla^{b,\lambda},
\qquad
H^{b,\lambda},
\qquad
k_i^{b,\lambda},
\qquad
d\mu^{b,\lambda},
\]
respectively. Then translation leaves both \( \int_\Sigma|\nabla H|_g^2\,d\mu \,\text{and}\, \int_\Sigma(k_1-k_2)^2\,d\mu \) unchanged, while under dilation by \( \lambda \),
\begin{equation}
\label{eq:coercive-scaling-laws}
\int_{\Sigma^{b,\lambda}}
|\nabla^{b,\lambda}H^{b,\lambda}|_{g^{b,\lambda}}^2\,d\mu^{b,\lambda}
=
\lambda^{-2}
\int_\Sigma|\nabla H|_g^2\,d\mu,
\end{equation}
and
\begin{equation}
\label{eq:curvature-difference-scale-invariant}
\int_{\Sigma^{b,\lambda}}
(k_1^{b,\lambda}-k_2^{b,\lambda})^2\,d\mu^{b,\lambda}
=
\int_\Sigma(k_1-k_2)^2\,d\mu.
\end{equation}
\end{lemma}

\begin{proof}
These identities follow by direct computation; see also \cite[p.~260]{huisken1984}.
\end{proof}

We now combine the preceding coercivity estimates and scaling laws with the Willmore evolution identity to prove the main theorem.

\section{Proof of the Main Theorem}
\begin{proof}[Proof of Theorem~\ref{main.thm}]
For Huisken's area-preserving normalized mean curvature flow of surfaces in \( \mathbb R^3 \), the normalized immersion satisfies
\[
\partial_{\tilde t}\mathbf{\tilde F}
=
-\tilde H\tilde\nu
+
\left(
\frac{1}{2|\tilde\Sigma_{\tilde t}|}
\int_{\tilde\Sigma_{\tilde t}}
\tilde H^2\,d\tilde\mu
\right)\mathbf{\tilde F}.
\]
The second term is the infinitesimal generator of a time-dependent ambient dilation. Since, for surfaces in \( \mathbb R^3 \), the Willmore energy is invariant under translations, reparameterizations, and constant ambient dilations, the scale factor \( \psi(t) \) does not contribute to the first variation of the Willmore energy. Hence the evolution of \( \tilde W \) is the same as for the unnormalized mean curvature flow, namely
\begin{equation}
\label{eq:normalized-W-evolution-ineq}
\frac{d}{d\tilde t}\tilde W(\tilde t)
=
-2
\int_{\tilde\Sigma_{\tilde t}}
|\tilde\nabla\tilde H|_{\tilde g}^2\,d\tilde\mu
+
\int_{\tilde\Sigma_{\tilde t}}
\tilde H^2(\tilde k_1-\tilde k_2)^2\,d\tilde\mu.
\end{equation}

Moreover, \( \tilde H=\tilde k_1+\tilde k_2, \, \tilde K=\tilde k_1\tilde k_2 \), so \( \tilde H^2 = (\tilde k_1-\tilde k_2)^2+4\tilde K. \) Since \( \tilde\Sigma_{\tilde t} \) is strictly convex, it is topologically a sphere. Hence \( \chi(\tilde\Sigma_{\tilde t})=2 \), and the Gauss--Bonnet theorem gives \( \int_{\tilde\Sigma_{\tilde t}} \tilde K\,d\tilde\mu = 4\pi. \) Consequently,
\begin{equation}
\label{eq:W-minus-16pi-defect}
\tilde W(\tilde t)-16\pi
=
\int_{\tilde\Sigma_{\tilde t}}
(\tilde k_1-\tilde k_2)^2\,d\tilde\mu.
\end{equation}

For the fixed value of \( \delta:=\frac{1}{16} \), let
\[
\varepsilon_{\mathrm c}
:=
\varepsilon(\bar g,\delta)
=
\varepsilon\left(\bar g,\frac1{16}\right)
>0
\]
be the smallness constant given by Lemma~\ref{lem:l-ge-2-coercivity}.

Let \( \varepsilon_0>0 \) and \( C_{\mathrm{gf}}>0 \) be the constants in Proposition~\ref{lem:gauge_fix_l01_strict}. Define
\begin{equation}
\label{eq:main-epsilon-star}
\varepsilon_*
:=
\frac12
\min\left\{
\varepsilon_0,
\frac{\varepsilon_{\mathrm c}}{C_{\mathrm{gf}}}
\right\}.
\end{equation}
Thus \( \varepsilon_* \) depends only on the fixed constants \( \bar g,\, \delta=\frac1{16},\, \varepsilon_0,\, C_{\mathrm{gf}} \), and in particular is independent of time.

By the smooth convergence of Huisken's normalized flow to the limiting round sphere, for all sufficiently large \( \tilde t \) the surface \( \tilde\Sigma_{\tilde t} \) can be written as a normal graph over \( S_R^2 \); see \cite[p.~328]{MazzeoPacard2011}. Hence there exists a constant
\[
\tilde T_{\mathrm{graph}}
=
\tilde T_{\mathrm{graph}}
(\Sigma_0,\varepsilon_*)
<\infty
\]
such that, for every \( \tilde t\ge\tilde T_{\mathrm{graph}} \),
\[
\tilde\Sigma_{\tilde t}
=
\{R(1+\tilde u(\omega,\tilde t))\omega:\omega\in S^2\},
\]
where
\begin{equation}
\label{eq:main-u-small}
\|\tilde u(\cdot,\tilde t)\|_{C^3(S^2)}
\le
\varepsilon_*.
\end{equation}

The graph function \( \tilde u \) need not be exactly orthogonal to the \( l=0 \) and \( l=1 \) spherical harmonics. We therefore first rescale to the unit-sphere scale by setting \( \widehat\Sigma_{\tilde t} := R^{-1}\tilde\Sigma_{\tilde t}. \) Then \( \widehat\Sigma_{\tilde t} = \{(1+\tilde u(\omega,\tilde t))\omega:\omega\in S^2\}. \)

Since \( \|\tilde u(\cdot,\tilde t)\|_{C^3(S^2)} \le \varepsilon_* < \varepsilon_0 \), we can apply Proposition~\ref{lem:gauge_fix_l01_strict} to \( \widehat\Sigma_{\tilde t}, \tilde t\ge\tilde T_{\mathrm{graph}} \). Hence, for \( \tilde t\ge\tilde T_{\mathrm{graph}} \), there exist \( a(\tilde t)\in\mathbb R^3, \, \rho(\tilde t)>0 \), such that
\[
\widehat\Sigma_{\tilde t}^{a,\rho}
:=
\rho(\tilde t)
\left(
\widehat\Sigma_{\tilde t}-a(\tilde t)
\right)
\]
can be written as
\[
\widehat\Sigma_{\tilde t}^{a,\rho}
=
\{(1+\check u(\omega,\tilde t))\omega:\omega\in S^2\},
\]
where
\[
\int_{S^2}\check u(\omega)\,d\omega=0,
\qquad
\int_{S^2}\check u(\omega)\omega\,d\omega=0.
\]
Equivalently, \( \check u \perp_{L^2(S^2)} (\mathcal H_0\oplus\mathcal H_1). \)

Moreover,
\[
|a(\tilde t)|
+
|\rho(\tilde t)-1|
\le
C_{\mathrm{gf}}
\|\tilde u(\cdot,\tilde t)\|_{C^1(S^2)},
\]
and
\[
\|\check u(\cdot,\tilde t)\|_{C^3(S^2)}
\le
C_{\mathrm{gf}}
\|\tilde u(\cdot,\tilde t)\|_{C^3(S^2)}.
\]
By \eqref{eq:main-u-small} and \eqref{eq:main-epsilon-star},
\[
\|\check u(\cdot,\tilde t)\|_{C^3(S^2)}
\le
C_{\mathrm{gf}}\varepsilon_*
\le
\frac12\varepsilon_{\mathrm c}
<
\varepsilon_{\mathrm c}
=
\varepsilon(\bar g,\delta).
\]
Therefore all the hypotheses of Lemma~\ref{lem:l-ge-2-coercivity} are satisfied for \( \widehat\Sigma_{\tilde t}^{a,\rho} \) for every \( \tilde t\ge\tilde T_{\mathrm{graph}} \). Hence, for \( \tilde t\ge\tilde T_{\mathrm{graph}} \), we obtain
\begin{equation}
\label{eq:main-gauged-coercivity}
\int_{\widehat\Sigma_{\tilde t}^{a,\rho}}
|\nabla^{a,\rho}H^{a,\rho}|_{g^{a,\rho}}^2\,d\mu^{a,\rho}
\ge
(4-\delta)
\int_{\widehat\Sigma_{\tilde t}^{a,\rho}}
(k_1^{a,\rho}-k_2^{a,\rho})^2\,d\mu^{a,\rho}.
\end{equation}

We now return to the original normalized surfaces \( \tilde\Sigma_{\tilde t} \). Notice that
\[
\widehat\Sigma_{\tilde t}^{a,\rho}
=
\rho(\tilde t)
\left(
R^{-1}\tilde\Sigma_{\tilde t}-a(\tilde t)
\right)
=
\frac{\rho(\tilde t)}{R}
\left(
\tilde\Sigma_{\tilde t}
-
R\,a(\tilde t)
\right).
\]
Define \( b(\tilde t) := Ra(\tilde t), \, \lambda(\tilde t) := \frac{\rho(\tilde t)}{R}. \) Then
\[
\widehat\Sigma_{\tilde t}^{a,\rho}
=
\lambda(\tilde t)
\left(
\tilde\Sigma_{\tilde t}
-
b(\tilde t)
\right).
\]
Moreover,
\[
\frac{|b(\tilde t)|}{R}
+
|R\lambda(\tilde t)-1|
\le
C_{\mathrm{gf}}
\|\tilde u(\cdot,\tilde t)\|_{C^1(S^2)}.
\]
Since
\[
\|\tilde u(\cdot,\tilde t)\|_{C^1(S^2)}
\longrightarrow0
\qquad
\text{as }\tilde t\to\infty,
\]
we obtain
\[
b(\tilde t)\to0,
\qquad
\lambda(\tilde t)\to\frac1R \qquad
\text{as }\tilde t\to\infty.
\]

By Lemma~\ref{lem:coercive-inequality-under-rigid-scaling}, \eqref{eq:main-gauged-coercivity} implies
\[
\int_{\tilde\Sigma_{\tilde t}}
|\tilde\nabla\tilde H|_{\tilde g}^2\,d\tilde\mu
\ge
\lambda(\tilde t)^2(4-\delta)
\int_{\tilde\Sigma_{\tilde t}}
(\tilde k_1-\tilde k_2)^2\,d\tilde\mu
\]
for every \( \tilde t\ge\tilde T_{\mathrm{graph}} \).

Set \( \sigma:=\frac1{4R^2}. \) Since \( \lambda(\tilde t)\to\frac1R \), there exists \( \tilde T_{\lambda} = \tilde T_{\lambda} (\Sigma_0,\varepsilon_*,R) \) satisfying \( \tilde T_{\mathrm{graph}} \le \tilde T_{\lambda} < \infty \), such that for every \( \tilde t\ge\tilde T_{\lambda} \),
\[
\left|
\lambda(\tilde t)^2-\frac1{R^2}
\right|
(4-\delta)
<
\frac{\sigma}{2}.
\]

For every \( \tilde t\ge\tilde T_{\lambda} \), we therefore have
\[
\lambda(\tilde t)^2(4-\delta)
\ge
\frac{4-\delta}{R^2}
-
\frac{\sigma}{2}.
\]
Using
\[
\frac{\delta}{R^2}
<
\frac{\sigma}{2},
\]
we obtain
\[
\lambda(\tilde t)^2(4-\delta)
>
\frac4{R^2}-\sigma.
\]
Consequently,
\[
\int_{\tilde\Sigma_{\tilde t}}
|\tilde\nabla\tilde H|_{\tilde g}^2\,d\tilde\mu
\ge
\left(
\frac4{R^2}-\sigma
\right)
\int_{\tilde\Sigma_{\tilde t}}
(\tilde k_1-\tilde k_2)^2\,d\tilde\mu.
\]
Using \eqref{eq:W-minus-16pi-defect}, this becomes
\begin{equation}
\label{eq:coercive-W-defect}
\int_{\tilde\Sigma_{\tilde t}}
|\tilde\nabla\tilde H|_{\tilde g}^2\,d\tilde\mu
\ge
\left(
\frac4{R^2}-\sigma
\right)
\left(
\tilde W(\tilde t)-16\pi
\right).
\end{equation}

Since the normalized flow converges smoothly to the round sphere of radius \( R \), we have \( \tilde H\to\frac2R \, \text{uniformly as }\tilde t\to\infty. \) For the fixed value \( \eta:=\frac1{4R^2} \), there exists \( \tilde T_H = \tilde T_H(\Sigma_0,R) <\infty \) such that
\begin{equation}
\label{eq:main-H-upper}
\tilde H^2
\le
\frac4{R^2}+\eta
\end{equation}
on \( \tilde\Sigma_{\tilde t} \) for every \( \tilde t\ge\tilde T_H \).

Finally, set
\[
\tilde T_0
:=
\max\left\{
\tilde T_{\lambda},
\tilde T_H
\right\}.
\]
Thus, before using the fixed choices above, its dependence may be displayed as \( \tilde T_0 = \tilde T_0 \left( \Sigma_0, R, \varepsilon_* \right). \) The limiting radius \( R \) is determined by the initial data \( \Sigma_0 \). Therefore the final time may ultimately be written simply as \( \tilde T_0=\tilde T_0(\Sigma_0). \)

For every \( \tilde t\ge\tilde T_0 \), both \eqref{eq:coercive-W-defect} and \eqref{eq:main-H-upper} hold. Combining them with \eqref{eq:normalized-W-evolution-ineq}, we obtain
\begin{align*}
\frac{d}{d\tilde t}\tilde W(\tilde t)
&\le
-2\left(
\frac4{R^2}-\sigma
\right)
\left(
\tilde W(\tilde t)-16\pi
\right)
+
\left(
\frac4{R^2}+\eta
\right)
\left(
\tilde W(\tilde t)-16\pi
\right)
\\
&=
\left(
-\frac4{R^2}
+
2\sigma
+
\eta
\right)
\left(
\tilde W(\tilde t)-16\pi
\right)
\\
&=-\frac{13}{4R^2}
\left(
\tilde W(\tilde t)-16\pi
\right).
\end{align*}

If \( \tilde\Sigma_{\tilde t} \) is not a round sphere, then \( \tilde W(\tilde t)-16\pi>0 \), and hence \( \frac{d}{d\tilde t}\tilde W(\tilde t)<0 \) for every \( \tilde t\ge\tilde T_0 \).

If \( \tilde\Sigma_{\tilde t} \) is a round sphere, then \( \tilde W(\tilde t)=16\pi \), and therefore \( \frac{d}{d\tilde t}\tilde W(\tilde t)=0. \) Consequently,
\[
\frac{d}{d\tilde t}\tilde W(\tilde t)\le0
\qquad
\text{for all }\tilde t\ge\tilde T_0,
\]
with equality only when \( \tilde\Sigma_{\tilde t} \) is a round sphere. This proves part~\textup{(ii)}.

We finally return to the unnormalized mean curvature flow. Recall that
\[
\tilde{\mathbf F}(\cdot,\tilde t)
=
\psi(t)\bigl(\mathbf F(\cdot,t)-p\bigr), \qquad \tilde t(t)
=
\int_0^t\psi^2(\tau)\,d\tau,
\]
where \( \psi(t)>0 \) for \( t\in [0,T) \). Since \( \frac{d\tilde t}{dt} = \psi^2(t)>0 \), the map \( t\mapsto\tilde t(t) \) is strictly increasing.

Since \( \tilde t(t)\to\infty \, \text{as }t\to T \), and \( t\mapsto\tilde t(t) \) is strictly increasing, there exists a unique \( t_0 = t_0(\Sigma_0,\tilde T_0) \in(0,T) \) such that \( \tilde t(t_0)=\tilde T_0>0. \) Since \( \tilde T_0=\tilde T_0(\Sigma_0) \), the time \( t_0 \) ultimately depends only on \( \Sigma_0 \).

For every \( t\ge t_0 \), \( \tilde t(t)\ge\tilde T_0. \) Differentiating \( W(t)=\tilde W(\tilde t(t)) \) with respect to \( t \), we obtain
\[
\frac{d}{dt}W(t)
=
\frac{d}{d\tilde t}\tilde W(\tilde t(t))
\frac{d\tilde t}{dt}
=
\psi^2(t)
\frac{d}{d\tilde t}\tilde W(\tilde t(t)).
\]
Since \( \psi(t)>0 \), the two derivatives have the same sign. Therefore, for every \( t\in[t_0,T) \), \( \frac{d}{dt}W(t)\le0. \) Moreover, equality at such a time occurs only when \( \tilde\Sigma_{\tilde t(t)} \), and hence \( \Sigma_t \), is a round sphere.

This proves part~\textup{(i)} and completes the proof.
\end{proof}

\section{Space Form Examples and Hawking-Mass Consequences}
The purpose of this section is twofold. First, the totally umbilical models show explicitly how the sign of the Willmore evolution changes with the ambient sectional curvature. Second, the monotonicity of Willmore energy provides the Hawking-mass monotonicity in the Euclidean case.

Throughout this section, the ambient manifold \( \breve{M}  \) and the initial hypersurface \( \Sigma_0 \) are assumed to satisfy the hypotheses of Huisken's convergence theorem \cite[Theorem~1.1]{Huisken1986Contracting}.

We write \( \breve R \) and \( \breve{\operatorname{Ric}} \) for the Riemann curvature tensor and Ricci tensor of the ambient manifold, respectively. Unless otherwise specified, \( g \), \( \nabla \), \( \Delta \), \( A \), \( H \), \( \nu \), and \( d\mu \) denote the metric, Levi--Civita connection, Laplace--Beltrami operator, second fundamental form, mean curvature, unit normal, and area element induced on \( \Sigma_t \), respectively. For an \( n \)-dimensional hypersurface, we write \( A^\circ:=A-\frac{H}{n}g \) for the trace-free second fundamental form. In dimension two, \( k_1 \) and \( k_2 \) denote the principal curvatures and \( K \) denotes the Gaussian curvature of \( \Sigma_t \).

\begin{lemma}[Preservation of total umbilicity in a space form]
\label{lem:umbilicity-preserved-space-form-huisken}
Let \( F:\Sigma_0\times[0,T)\longrightarrow(\breve M^{n+1},\breve g) \) be a smooth mean curvature flow satisfying \( \partial_tF=-H\nu \). Define \( \Sigma_t:=F(\Sigma_0,t) \). Assume that \( (\breve M^{n+1},\breve g) \) has constant sectional curvature \( c \). If
\[
A^\circ(\cdot,0)=0,
\]
then
\[
A^\circ(\cdot,t)\equiv0
\]
for as long as the smooth solution exists.
\end{lemma}

\begin{proof}
Set \( Q:=|A|^2-\frac1nH^2=|A^\circ|^2. \) Huisken's evolution equation for \( Q \) in an ambient Riemannian manifold \cite{Huisken1986Contracting}, together with the space form identity
\[
\breve R_{\alpha\beta\gamma\delta}
=
c\left(
\breve g_{\alpha\gamma}\breve g_{\beta\delta}
-
\breve g_{\alpha\delta}\breve g_{\beta\gamma}
\right)
\]
(see, for example, \cite[p.~121]{MikesStepanova2014}), gives
\[
\partial_tQ
=
\Delta Q
-2\left(
|\nabla A|^2-\frac1n|\nabla H|^2
\right)
+
2(|A|^2-nc)Q.
\]
Since \( A^\circ=A-\frac{H}{n}g \,\text{and}\, \nabla g=0 \), we have \( |\nabla A^\circ|^2 = |\nabla A|^2-\frac1n|\nabla H|^2. \) Hence
\[
\partial_tQ
=
\Delta Q
-2|\nabla A^\circ|^2
+
2(|A|^2-nc)Q
\le
\Delta Q+2(|A|^2-nc)Q.
\]
The coefficient \( 2(|A|^2-nc) \) on the right-hand side is bounded on every compact time interval on which the smooth flow exists. Since \( Q(\cdot,0)=0 \), the parabolic maximum principle yields \( Q(\cdot,t)\equiv0. \) Thus \( A^\circ(\cdot,t)\equiv0 \).
\end{proof}

Having established the preservation of total umbilicity, the Willmore evolution simplifies considerably. We now compute it explicitly in dimension two.

\begin{example}[Willmore energy along totally umbilical mean curvature flow in a space form]
\label{lem:willmore-energy-umbilic-space-form}

Let \( \{\Sigma_t^2\}_{0\le t<T} \) be a smooth mean curvature flow of closed, connected, totally umbilical surfaces in a three-dimensional space form \( (\breve M^3,\breve g) \) of constant sectional curvature \( c \). Then the Willmore energy \( W(t) := \int_{\Sigma_t}H^2\,d\mu \) satisfies
\[
\frac{d}{dt}W(t)
=
4c\,W(t).
\]
Consequently, the Willmore energy is constant for \( c=0 \), nondecreasing for \( c>0 \), and nonincreasing for \( c<0 \).
\end{example}

\begin{proof}
By Lemma~\ref{lem:umbilicity-preserved-space-form-huisken}, \( A^\circ\equiv0, \, A=\frac{H}{2}g, \, |A|^2=\frac12H^2. \) Since \( \breve{\operatorname{Ric}}(\nu,\nu)=2c \) in a three-dimensional space form, the standard evolution equations \cite{Huisken1986Contracting} give
\[
\frac{d}{dt}\int_{\Sigma_t}H^2\,d\mu
=
-2\int_{\Sigma_t}|\nabla H|^2\,d\mu
+
4c\int_{\Sigma_t}H^2\,d\mu.
\]

It remains to show that \( H \) is constant on each connected component. Since the ambient manifold is a space form, the Codazzi equation reduces to \( \nabla_i h_{jk}=\nabla_j h_{ik}. \) Using \( h_{jk}=\frac12Hg_{jk} \), we obtain \( (\nabla_iH)g_{jk} = (\nabla_jH)g_{ik}. \) Contracting with \( g^{jk} \) gives \( 2\nabla_iH=\nabla_iH \), and therefore \( \nabla H\equiv0. \) Hence
\[
\frac{d}{dt}\int_{\Sigma_t}H^2\,d\mu
=
4c\int_{\Sigma_t}H^2\,d\mu,
\]
which proves the claim.
\end{proof}

The same model also makes the behavior of the Hawking mass completely explicit. The following corollary records the corresponding evolution and its dependence on the sign of \( c \).

\begin{corollary}[Hawking mass along the totally umbilical flow]
\label{cor:hawking-mass-umbilic-space-form}
Let \( \Sigma_t^2\subset(\breve M^3,\breve g) \) be as in Example~\ref{lem:willmore-energy-umbilic-space-form}. Set \( \mathcal A(t):=|\Sigma_t|. \) Then for the Hawking mass
\[
m_H(\Sigma_t)
=
\sqrt{\frac{\mathcal A(t)}{16\pi}}
\left(
1-\frac{W(t)}{16\pi}
\right),
\]
it holds that
\[
\frac{d}{dt}m_H(\Sigma_t)
=
-\frac{3c\mathcal A(t)^{1/2}}{8\pi^{3/2}}
\left(
4\pi-c\mathcal A(t)
\right).
\]
Thus the Hawking mass is identically zero for \( c=0 \), nonincreasing for \( c>0 \), and nondecreasing for \( c<0 \).
\end{corollary}

\begin{proof}
The area evolution and Example~\ref{lem:willmore-energy-umbilic-space-form} imply that \( \mathcal A'(t)=-W(t), \, \frac{d}{dt}W(t)=4cW(t). \)
Since it is totally umbilical, \( k_1=k_2=\frac{H}{2}. \) The Gauss equation gives \( K=c+k_1k_2 = c+\frac14H^2 \) (see, for example, \cite[p.~81]{MeeksPerezRos2008}). Integrating and applying the Gauss--Bonnet theorem, we obtain
\[
4\pi
=
c\mathcal A(t)+\frac14W(t).
\]
Therefore
\[
1-\frac{W(t)}{16\pi}
=
\frac{c\mathcal A(t)}{4\pi},
\]
and hence
\[
m_H(\Sigma_t)
=
\frac{c\mathcal A(t)^{3/2}}{16\pi^{3/2}}.
\]
Differentiating and using
\[
\mathcal A'(t)
=
-W(t)
=
-4\left(4\pi-c\mathcal A(t)\right)
\]
gives
\[
\frac{d}{dt}m_H(\Sigma_t)
=
-\frac{3c\mathcal A(t)^{1/2}}{8\pi^{3/2}}
\left(
4\pi-c\mathcal A(t)
\right).
\]
Since \( 4\pi-c\mathcal A(t)=\frac14W(t)\ge0 \), the corollary is proved.
\end{proof}

The preceding computation in space forms highlights the role of ambient curvature. We now return to the Euclidean setting, where the main theorem supplies the missing sign in the Hawking-mass evolution.

\begin{proof}[Proof of Corollary~\ref{cor:hawking-mass-euclidean-late-time-intro}]
For the Hawking mass
\[
m_H(\Sigma_t)
=
\sqrt{\frac{\mathcal A(t)}{16\pi}}
\left(
1-\frac{W(t)}{16\pi}
\right),
\]
we have
\begin{equation}
\label{hawking.mass.derivative}
\frac{d}{dt}m_H(\Sigma_t)
=
\sqrt{\frac{\mathcal A(t)}{16\pi}}
\left[
-\frac{W(t)}{2\mathcal A(t)}
\left(
1-\frac{W(t)}{16\pi}
\right)
-
\frac{W'(t)}{16\pi}
\right].
\end{equation}
By the Willmore inequality, we have \( W(t)\ge16\pi \), while Theorem~\ref{main.thm} gives \( \frac{d}{dt}W(t)\le0 \) for all \( t \) in \( [t_0, T) \). Thus both terms on the right-hand side of \eqref{hawking.mass.derivative} are nonnegative, and hence \( \frac{d}{dt}m_H(\Sigma_t)\ge0 \) in the same asymptotic regime.

If equality holds at such a time, then \( W(t)=16\pi \), and therefore \(\Sigma_t \) is a round sphere by the equality case of the Willmore inequality.
\end{proof}

\bibliographystyle{plain}
\bibliography{main}

\end{document}